\documentclass[10pt]{amsart}
\usepackage[a4paper, margin=1in]{geometry}
\usepackage[utf8]{inputenc}
\usepackage{mathtools}
\usepackage{amsmath}
\usepackage{amsthm}
\usepackage{amssymb}
\usepackage{mathrsfs}
\usepackage{enumitem}
\usepackage{tikz-cd}
\usepackage{dsfont}
\newcommand{\spa}[2]{\operatorname{span}_{#1}\left\{#2\right\}}

\newcommand{\Z}{\mathbb{Z}} 
\newcommand{\N}{\mathbb{N}} 
\newcommand{\gen}[1]{\langle #1 \rangle}
\newcommand{\paren}[1]{\left\{ #1 \right\}}
\newcommand{\p}[1]{\left( #1 \right)}

\newcommand{\F}{\mathbb{F}}
\newcommand{\Q}{\mathbb{Q}}

\newcommand{\gr}{\operatorname{gr}} 

\newcommand{\FG}{{\F_p\dbl G\dbr}}

\newcommand{\ur}{\mathfrak{u}}

\newcommand{\CVec}{{}_{\mathbb{F}_p}\mathbf{cvec}}

\newcommand{\iid}{\mathrm{id}}

\numberwithin{equation}{section}

\newtheorem{theorem}{Theorem}[section]

\newtheorem{lemma}[theorem]{Lemma}
\newtheorem{fact}[theorem]{Fact}
\newtheorem{corollary}[theorem]{Corollary}
\newtheorem{proposition}[theorem]{Proposition}
\theoremstyle{definition}
\newtheorem{definition}[theorem]{Definition}
\newtheorem{remark}[theorem]{Remark}
\newtheorem{example}[theorem]{Example}

\newtheorem{question}[theorem]{Question}

\theoremstyle{plain}
\newtheorem{teoInt}{Theorem}

\newcommand{\caB}{\mathcal{B}}
\newcommand{\caW}{\mathcal{W}}
\newcommand{\ucaW}{\underline{\caW}}
\newcommand{\End}{\mathrm{End}}
\newcommand{\op}{\mathrm{op}}

\newcommand{\eps}{\varepsilon}

\newcommand{\euV}{\mathscr{V}}
\newcommand{\euE}{\mathscr{E}}

\newcommand{\eua}{\mathbf{a}}
\newcommand{\ddGamma}{\ddot{\Gamma}}
\newcommand{\dbl}{[\![}
\newcommand{\dbr}{]\!]}
\newcommand{\bool}{\mathbf{bool}}
\newcommand{\kernel}{\mathrm{ker}}
\newcommand{\cl}{\mathrm{cl}}
\newcommand{\uomega}{\underline{\omega}}
\newcommand{\uD}{\mathbb{D}}
\newcommand{\Syl}{\mathrm{Syl}}

\newcommand{\propp}{\mathrm{pro-}p}
\newcommand{\image}{\mathrm{im}}
\newcommand{\chrs}{\mathrm{char}}
\newcommand{\euA}{\mathscr{A}}
\newcommand{\euP}{\mathscr{P}}
\newcommand{\argu}{\hbox to 1.5ex{\hrulefill}}

\title{Generalising a Theorem of Lichtman}
\author{Giorgio Leoni}
\author{Conchita Martínez Pérez}
\author{Thomas Weigel}
\date{today}



\thanks{\noindent
The second named author is partially supported by the Spanish Government PID2021-126254NB-I00 and PID2024-155800NB-C32 and also by Departamento de Ciencia, Universidad y Sociedad del 
Conocimiento del Gobierno de Arag{\'o}n (grant code: E22-23R: ``{\'A}lgebra y Geometr{\'i}a'').
The first and third named authors are members of the
Gruppo Nazionale per le Strutture Algebriche, Geometriche e le loro Applicazioni
(GNSAGA), which is part of the Istituto Nazionale di Alta Matematica (INdAM)}

\begin{document}

\begin{abstract}
We show that under a suitable additional hypothesis the restricted Zassenhaus $\F_p$-Lie algebra or the rational Magnus Lie algebra of a free amalgamated product is the free amalgamated product of the corresponding Lie algebras of the factors. This generalises a Theorem of A.I.\,Lichtman, who proved the analoguous statement for free products. Our conditions include the case when the amalgamated product is a retract in both factors. As a by-product, we show that a free product of residually torsion free nilpotent groups amalgamated along retracts is also residually torsion free nilpotent and obtain also some results on cohomological completeness. In the final sections we apply our main results to two recently raised open questions.
\end{abstract}

\maketitle

\section{Introduction}\label{sec:introduction}
As Efim Zelmanov's solution of the restricted Burnside problem has shown (cf. \cite{MVL:resburn}), 
Lie theoretic methods are undoubtedly one of the most powerful tools in group theory.  Nevertheless, few results are known (cf. \cite{FalkRandell}, \cite{Lichtman}, \cite{MRW}) which can be used for an explicit computation of the (rational) Magnus Lie algebra or of the restricted Zassenhaus $\F_p$-Lie algebra of a (pro $p$-)group $G$. In this context, we call
the $\N$-graded Lie ring $\gr_\bullet(G)=\oplus_{k\in\N}\gamma_k(G)/\gamma_{k+1}(G)$ the \emph{Magnus} Lie ring of $G$, where $\gamma_k(G)$ is the $k$-th term of the descending central series of $G$. Moreover, we use $\gr^{[0]}_\bullet(G)=\gr_\bullet(G)\otimes_\Z\Q$ to indicate the \emph{rational Magnus Lie algebra} of $G$. 

For a prime number $p$, we denote by $\gr^{[p]}_\bullet(G)=\oplus_{k\in\N}D_k(G,\F_p)/D_{k+1}(G,\F_p)$ the \emph{restricted Zassenhaus} $\F_p$-Lie algebra of $G$ (cf. \cite{Zassenhaus}). Here, for a commutative ring $F$,  $D_k(G,F)$  is the $k$-th dimension subgroup of the $F$-group algebra $F G$, i.e., the group consisting of those elements $g\in G$ for which $1-g$ lies in $(\omega_G^F)^k$, where $\omega_G^F$ denotes the augmentation ideal of $F G$.  It turns out that the rational Magnus Lie algebra can also be obtained from the rational dimension subgroups $D_k(G,\Q)$ making our notations consistent. Consequently, we are able to deal with these two types of algebras at the same time.

Dimension subgroups with integer coefficients are rather mysterious; however, thanks to results by Jennings, Lazard and Quillen (see Section \ref{sec:dcs}) when one works with coefficients in a field, they can be described in terms of the group only. 

Let $p$ be either zero or a prime. One of the few available results that allows  an explicit computation of the Lie algebra $\gr^{[p]}_\bullet(G)$ in some cases is a theorem by 
Lichtman \cite{Lichtman} who proved that the functor $\gr^{[p]}_\bullet(-)$  respects free products. 

 On the other hand  Falk and Randell (cf. \cite{FalkRandell}) showed that even the Magnus Lie ring functor $\gr_\bullet(-)$ behaves well for \emph{almost direct products} which are semidirect products $Q\ltimes N$ where the action of the group $Q$ is trivial on the abelianization of $N$. It is an easy observation that their result yields the same for $\gr^{[p]}_\bullet(-)$ (Section \ref{sec:proofAB}) with a $p$-local version of almost direct products. Our first main result can be seen as a generalization of both Lichtman's theorem (to free products with amalgamation) and Falk-Randell's (to certain HNN extensions). More explicitly, we show

\begin{teoInt} \label{thm:graded} Let $p$ be either zero or a prime. 
\begin{itemize}
\item[\textup{(i)}] Let $G_1$, $G_2$ be groups with a common subgroup $H$ which is strictly $p$-embedded in both. Then
\begin{equation}
    \label{eq:iden1}
\gr^{[p]}_\bullet(G)=\gr^{[p]}_\bullet(G_1)\ast_{\gr^{[p]}_\bullet(H)}\gr^{[p]}_\bullet(G_2)
\end{equation}
where $G=G_1\ast_H G_2$.

\item[\textup{(ii)}] Let $H$ be a strictly $p$-embedded subgroup of $G_1$ and let $\varphi\colon H\to H$ be an almost $p$-trivial automorphism. Let $G=G_1\ast_{H,t}$ be the associated HNN-extension. Then
\begin{equation}
    \label{eq:ident2}
\gr^{[p]}_\bullet(G)=\gr^{[p]}_\bullet(G_1)\ast_{\gr^{[p]}_\bullet(H),t}
\end{equation}
\end{itemize}
 \end{teoInt}
 
 Here, we say that the subgroup $H$ of $G$ is \emph{strictly $p$-embedded} if the inclusion map $\iota\colon H\to G$ induces an injective map $\gr_\bullet^{[p]}(\iota)\colon\gr^{[p]}_\bullet(H)\to
\gr^{[p]}_\bullet(G)$  (see Section \ref{sec:dcs}). This is the case if $H$ is a retract. 
Similarly, an automorphism $\varphi\colon H\to H$ is said to be
$p$-almost trivial, if $\gr^{[p]}_\bullet(\varphi)\colon
\gr^{[p]}_\bullet(H)\to\gr^{[p]}_\bullet(H)$ is the identity on $\gr^{[p]}_\bullet(H)$.
The conclusion of Theorem A appeared in the PhD thesis of the first named author in the context of pro-$p$ groups \cite{GiorgioThesis}.
 Sjogren showed in \cite{Sjogren} that the quotients $D_i(G,\Z)/\gamma_i(G)$ are torsion groups.
 So if the groups $G_1$ and $G_2$ are finitely generated and $\gr_\bullet(G_1)$ and $\gr_\bullet(G_2)$ are torsion free, then the dimension subgroups coincide with the terms of the descending central series. Hence in this case one may use Theorem \ref{thm:graded} to deduce a similar result for the Magnus Lie ring $\gr_\bullet(G)$.

 The proof of Theorem \ref{thm:graded}  has the following consequence

 \begin{teoInt} \label{teo:strictly}
 Let $p$ be either zero or a prime and let $G$ be either a free amalgamated product or an almost $p$-trivial 
 \textup{HNN}-extension such that the edge groups are strictly $p$-embedded in the vertex groups. Then the vertex groups are strictly $p$-embedded in $G$. Moreover, if the vertex groups are residually $[p]$, then  $G$ is residually $[p]$.
 \end{teoInt}
 The notion of being \emph{residually $[p]$} is defined in Section \ref{sec:dcs} and is equivalent to being residually a nilpotent $p$-group with bounded torsion if $p$ is a prime, or being residually torsion free nilpotent, if $p=0$. In particular, this theorem implies that free amalgamated products along retracts of residually torsion free nilpotent groups are also residually torsion free nilpotent.  In Section \ref{sec:rtfn} we give a direct group theoretical proof of this fact, which might be known to experts, but we have not been able to find any proof in the literature.

Amalgamating along almost $p$-retracts also behaves well with respect to cohomological properties. Recall that a group $G$ is said to be 
\emph{cohomologically $p$-complete} (cf. \cite{Serre}) if the natural map from the continuous cohomology with $\F_p$ coefficients of its pro-$p$ completion to the ordinary cohomology with $\F_p$ coefficients of the group is an isomorphism (see Section \ref{sec:chpc}). In \cite{AF}, Aschenbrebber and Friedl give conditions on a graph of groups $\mathcal{G}$ that imply that the fundamental group $\pi(\mathcal{G})$ is cohomologically $p$-complete. See also \cite{JS}, where the case of \emph{algebraically clean} graphs of groups is considered (these are graphs of groups where vertex and edge groups are free and edge groups are free factors of vertex groups). Here we use results in \cite{AF} in order to show:

\begin{teoInt}\label{thm:pcomplete} Let $p$ be a prime and let $G$ be either a free amalgamated product or an almost $p$-trivial \textup{HNN}-extension of finitely 
generated groups such that the vertex groups are residually $p$-finite and cohomologically $p$-complete, and that the edge groups are almost $p$-retracts in the vertex groups. Then $G$ is cohomologically $p$-complete. 
\end{teoInt}

{In the last two sections of the paper we consider some applications of our main results.
First, we consider twisted RAAGs and use Theorem \ref{thm:graded} to describe their 2-Zassenhaus Lie algebras providing an affirmative answer to  \cite[Question 4.3]{Blumer}.}
Secondly, we apply a pro-$p$ version of Theorem~\ref{thm:graded} to describe the
$p$-Zassenhaus restricted Lie algebra of some twisted
right-angled Artin pro-$p$ groups providing an affirmative answer to \cite[Question 1.4]{BQW}. 
More applications can be found in a forthcoming paper \cite{ELMP}, where we use Theorem \ref{thm:graded} to describe the $p$-Zassenhaus Lie algebras of a certain family of Artin groups that includes all even Artin groups of FC-type and Theorem \ref{thm:pcomplete} to show that residually $p$-finite groups in that family are also cohomologically $p$-complete.

\section{Amalgamated free products of algebras}
\label{s:alg}
From now on we use the notion \emph{algebra}
as a short form for \emph{associative algebra with 1}.

\begin{definition}\label{definition_FAP} Let $\F$ be a field and let $A_1$, $A_2$ and $C$ be $\F$-algebras with morphisms $\lambda_1\colon C \to A_1$ and $\lambda_2\colon C \to A_2$. The \emph{free product} of $A_1$ and $A_2$ \emph{amalgamated along $C$} is the quotient algebra 
\begin{align*}
    A_1 \ast_{C} A_2\coloneqq T(A_1,A_2)/I
\end{align*}
where $T(A_1,A_2)=\bigoplus_n T_n({A_1},{A_2})$, $T_n({A_1},{A_2})$ is the sum of all tensor products of n terms over $A_1\cup A_2$ and 
$I\subseteq T(A_1,A_2)$ is the ideal generated by the elements
\begin{align*}
    \paren{\,a \otimes b-ab, \hspace{0.5em} \lambda_{1}(c)-\lambda_{2}(c) \mid a,b \in A_i,\hspace{0.5em} i=1,2,\hspace{0.5em} c \in C\,}.
\end{align*}
Multiplication is defined as the $\F$-bilinear map given by juxtaposition of terms.
\end{definition}

Let $\psi_i\colon A_i \hookrightarrow T(A_1,A_2) \twoheadrightarrow A_1 \ast_{C} A_2$ be the natural morphisms for $i=1,2$. Then - by definition - there is a commuting diagram
    \begin{equation*}
        \begin{tikzcd}
            C \arrow[r, "\lambda_1"] \arrow[d, "\lambda_2"] & A_1 \arrow[d, "\psi_1"] \\
A_2 \arrow[r, "\psi_2"]                & A_1 \ast_{C} A_2,         \end{tikzcd}
    \end{equation*}
i.e., $\psi_1\circ\lambda_1=\psi_2\circ\lambda_2$ and $A = A_1 \ast_{C} A_2$ is characterized by the following universal property: If $D$ is an $F$-algebra and $\phi_i: A_i \to D$, $i=1,2$,  are morphisms such that $\phi_1\circ \lambda_1=\phi_2 \circ \lambda_2$, then there exists a unique morphism $\phi: A \to D$ such that $\phi_1=\phi \circ \psi_1$ and $\phi_2=\phi \circ \psi_2$.

In general, the condition that $\lambda_1$ and $\lambda_2$ are injective does not imply that the morphisms $\psi_1$ and $\psi_2$ are. 
Nevertheless, this is true whenever both $A_1$ and $A_2$ are free right $C$-modules as one is able to establish a normal form-type result which will be described in the sequel. Let $C$, $A_1$ and $A_2$ be $\F$-algebras, and let $\lambda_1: C \to A_1$ and $\lambda_2: C \to A_2$ be injective morphisms of $\F$-algebras. For simplicity we identify the element
$1_{A_1}$ with $1_{A_2}$ and denote it simply by $1$. We also assume that $\lambda_i\colon C\to A_i$ are given by inclusion and
 that $A_i$, $i=1,2$ are free right $C$-modules
 with $C$-basis $\caB_i$ satisfying $1 \in \caB_i$. 
Put $\caB_i^\sharp=\caB\setminus\{1\}$. Consider the tuples
\begin{align*}
    (v_{1}, \dots, v_{n})\in\caB_{i_1}^\sharp\times\cdots\times\caB_{i_n}^\sharp
\end{align*}
for $n \in \mathbf{\Z}_{\geq 1}$, $i_k \in \paren{1,2}$ for $k=1,\dots,n$ and $()=1$ for $n=0$. 
Then $(v_{1}, \dots, v_{n})$ is said to be reduced, if 
either $n=0$ or $n\geq 1$ and $i_k\neq i_{k+1}$ for all $k=1,\dots,n-1$. 
Denote by $\caW$ the set of all reduced tuples and let $\ucaW$
denote the free right $C$-module over the set $\caW$. By construction, one has
 a map $\Psi_\circ\colon \caW \to A_1\ast_{C} A_2$ given by
\begin{align*}
    \Psi_\circ(v_1, \dots, v_n)
    =\psi_{i_1}(v_1)\dots \psi_{i_n}(v_n),
\end{align*}
which yields a homomorphism of right $C$-modules $\Psi\colon\ucaW\to A_1\ast_{C} A_2$.
The following property may be considered as a normal form theorem.

\begin{proposition}\label{prop:normalform}
The homomorphism of right $C$-modules $\Psi\colon\ucaW\to A_1\ast_{C} A_2$.
is an isomorphism.
\end{proposition}

\begin{proof}
As $\Psi$ is certainly surjective, it suffices to show that $\Psi$ is injective.
By hypothesis, for $y,a\in A_i$, $i=1,2$, there exist unique elements $c_x^{(i)}(y,a)\in C$, $x\in\caB_i$,
satisfying
\begin{align}
    y\cdot a&=c_1^{(i)}(y,a)+\sum_{x\in\caB_i^\sharp} x\cdot c_x^{(i)}(y,a).\label{eq:disp1}\\
\intertext{Moreover, $S^{(i)}(a,y)=\{\,x\in\caB_i\mid c_x^{(i)}(y,a)\not=0\,\}$ is a finite set.
In particular, for $y,a, b\in A_i$, $i=1,2$, one obtains}
 y\cdot a\cdot b&=\sum_{x\in\caB_i} x\cdot c_x^{(i)}(y,a)\cdot b.\label{eq:disp2}\\
 &=\sum_{x\in\caB_i} x\cdot c_x^{(i)}(y,ab).\label{eq:disp3}
\end{align}
For $a\in A_i$, $(v_1,\ldots,v_n)\in\caW$ and $c\in C$,  put
\begin{equation}\label{eq:disp4}
\phi_i(a)((v_1,\ldots,v_n)c)=
\begin{cases}
    (v_1,\ldots,v_{n-1})\cdot c_1^{(i)}(v_n c,a)+\\
    \sum_{x\in\caB_i^\sharp} (v_1,\ldots,v_{n-1},x)\cdot c_x^{(i)}(v_n c,a)&\ \text{if $i_n=i$,}\\
    (v_1,\ldots,v_n)\cdot c_1^{(i)}(c,a)+\\
    \sum_{x\in\caB_i^\sharp} (v_1,\ldots,v_{n},x)\cdot c_x^{(i)}(c,a)&\ \text{if $i_n\not=i$.}
\end{cases}
\end{equation}

Then \eqref{eq:disp4} defines a map 
$\phi_i(a)\in\End_{\F}(\ucaW)$ which coincides with right multiplication by 
$a\in A_i$ under the map $\Psi$, i.e., $\Psi(\phi_i(a)(\omega))=\Psi(\omega)\cdot \lambda_i(a)$
for $a\in A_i$ and $\omega\in \ucaW$, $i=1,2$.
Moreover,  from \eqref{eq:disp2} and \eqref{eq:disp3} one concludes by a lengthy but elementary calculation, that 
$\phi_i\colon A_i\to\End_{\F}(W)^{\op}$, $i=1,2$,  are $\F$-algebra homomorphisms satisfying
$\phi_1\vert_C=\phi_2\vert_C$. Hence, by the universal property, $\phi_i$, $i=1,2$,
induce a unique homomorphism $\phi\colon A_1\ast_C A_2\longrightarrow \End_{\F}(\ucaW)^{\op}$
satisfying $\phi\circ\psi_i=\phi_i$, $i=1,2$. By construction, one has
\begin{equation}
    \label{eq:disp5}
\phi(\Psi(\omega))(1)=\omega
\end{equation}
for all $\omega\in\caW$ and thus for all $\omega\in\ucaW$ showing that $\Psi$ is injective.
\end{proof}

\begin{corollary}\label{corollary_injective}
Under the previous hypotheses, the maps $\psi_1$ and $\psi_2$ are injective.
\end{corollary}

\begin{proof}
By construction, one has for $a\in A_i$ that 
$\Psi(\psi_i(a))=\sum_{x\in\caB_i} x\cdot c_x^{(i)}(1,a)$ which is $0$ if, and only if, $a=0$.
\end{proof}

\section{Filtered and graded algebras}
\label{s:filgr}

\begin{definition} An $\F$-algebra (with identity) $A$ together with a chain of ideals 
$$A=F^0A\supseteq F^1A\supseteq\ldots\supseteq F^nA\supseteq F^{n+1}A\supseteq\ldots$$
satisfying that for all $a\in F^nA$ and $b\in F^mA$ one has $ab\in F^{n+m}A$
is said to be a \emph{filtered $\F$-algebra}.
If $\bigcap_{n\geq1} F^nA=0$, then we say that the filtered $\F$-algebra $F^\bullet A$ is \emph{separated}.
Let $F^\bullet A$ and $F^\bullet B$ be filtered $\F$-algebras. A homomorphism $\varphi\colon A\to B$ of $\F$-algebras is said to be \emph{filtered} if for all $n\geq 0$, $\varphi(F^nA)\subseteq F^nB$. We say that $\varphi$ is \emph{strictly filtered} or just \emph{strict}, if for any $n\geq 0$, one has $\varphi(A)\cap F^nB=\varphi(F^nA)$.

Let $F^\bullet A$ be a filtered $\F$-algebra. A right $A$-module $M$ together with a chain of submodules 
$$M=F^0M\supseteq F^1M\supseteq\ldots\supseteq F^nM\supseteq F^{n+1}M\supseteq$$
is said to be \emph{filtered},
if for all $a\in F^nA$ and $x\in F^mM$ one has $xa\in F^{n+m}M$.
\end{definition}

Quite often it will be useful to think of filtrations in terms of valuations. A 
\emph{valuation} on the $\F$-algebra $A$ is a map
$\nu_A\colon A\to\Z_{\geq 0}\cup\{\infty\}$ satisfying for every $a,b\in A$
\begin{itemize}
\item[(i)] $\nu_A(a-b)\geq\inf(\nu_A(a),\nu_A(b))$,

\item[(ii)] $\nu_A(ab)\geq\nu_A(a)+\nu_A(b)$,

\item[(iii)] $\nu_A(1)=0$.
\end{itemize}
Given a valuation $\nu_A$ on $A$, we may define a filtration $F^\bullet A$ on $A$ by
$$F^nA=\{a\in A\mid\nu_A(a)\geq n\}.$$
Conversely, any filtration $F^\bullet A$ on $A$ yields the valuation
$$\nu_A(a)=\sup\{n\mid a\in F^nA\}.$$
Furthermore, if $A$, $B$ are filtered $\F$-algebras with valuations $\nu_A$, $\nu_B$, a homomorphism $\varphi\colon A\to B$ is filtered if, and only if, $\nu_B(\varphi(a))\geq\nu_A(a)$, and strictly filtered if, and only if, for every $a\in A$ there is some $a'\in A$ such that $\varphi(a)=\varphi(a')$ and $\nu_B(\varphi(a'))=\nu_A(a')$. 

In the next definition we follow \cite[2.1.16]{Lazard2}. 
Here and throughout the paper, the notion module will stand for \emph{right module}.
\begin{definition}
    Let $F^\bullet A$ be a filtered $\F$-algebra and let $F^\bullet M$ be a filtered $F^\bullet A$-module. We say that a set of elements $\Omega=\{ x_i\mid i \in I\,\}$ of $M$ is \emph{filtered $A$-free} or simply \emph{filtered free} if $1\in\Omega$ and 
    \begin{itemize}
\item[(i)]    for any family $(a_i)_{i\in I}$ of elements of $A$ where almost all $a_i$'s are zero,
    \begin{align*}
        \nu_M\p{\sum_{i \in I} x_ia_i}=\inf\{\,\nu_A(a_i)+\nu_M(x_i)\mid i\in I\,\}.
    \end{align*}
    
    \item[(ii)] the family $(x_i)_{i \in I}$ is $A$-linearly independent.
    \end{itemize}
    A module that is generated by a filtered free family is said to be \emph{filtered free}.
\end{definition}

\begin{remark} If $A$ and $M$ are separated, i.e., if $\cap F^nA=0=\cap F^nM$, then (i) implies (ii). To see this, note that if $\sum_{i \in I} x_ia_i=0$, then $\nu_M(\sum_{i \in I} x_ia_i)=+\infty$ so $\nu_M(x_i)+\nu_A(a_i)=+\infty$ for all  $i \in I$.
 Since $A$ is separated, if $a_i\neq 0$ for some $i$ then $\nu_A(a_i)<+\infty$. It follows that $\nu_M(x_i)=+\infty$. But as $M$ is also separated, $x_i=0$.
\end{remark}

\begin{definition} Let $C$, $A$ be filtered algebras. A  filtered morphism $\phi: C \to A$ induces a filtered $C$-module structure on $A$ by $a \cdot c \coloneqq a\phi(c)$. 
We say that $\phi: C \to A$ is \emph{filtered free} if $A$ is filtered free as a filtered $C$-module.
\end{definition}

\begin{remark} Assume that $C$ and $A$ are filtered .
    Then, if $\phi: C \to A$ is filtered free, it is strict. 
 In order to see this, take  $c \in C$. For the element  $1$ in the basis we have
    \[\nu_A(\phi(c))= \nu_A(1\cdot c) =\nu_C(c)+\nu_A(1)=\nu_C(c). \]
   If moreover $C$ and $A$  are separated, then $\phi$ is also injective: take $c \in C$ such that $\phi(c)=0$ . Then $\infty
        =\nu_A(\phi(c))=\nu_C(c)$ thus $c=0$. 
        \end{remark}

    If $A$ is a filtered algebra and $M$ is a free $A$-module with basis $(x_i)_{i\in I}$, we may endow $M$ with a filtered $A$-module structure by choosing as $\nu_M(x_i)$ an element in $\Z_{\geq0}\cup\{+\infty\}$  for every $i \in I$, and defining for any $m \in M$ with $m=\sum_i x_i a_i$,
    \begin{align*}
        \nu_M(m)\coloneqq\inf(\nu_M(x_i)+\nu_R(a_i)).
    \end{align*}
    With such filtration, $(x_i)_{i\in I}$ is a filtered $A$-basis for $M$ and if the values of $\nu_M(x_i)$ are all nonnegative integers, then $M$ is separated.

    If additionally $M$ is an algebra and we choose $\nu_M(x_i)$ so that, for any $i$ and $j$,
    \begin{align*}
        \nu_M(x_ix_j)\geq \nu_M(x_i)+\nu_M(x_j)
    \end{align*}
    then $M$ is a filtered algebra.

This implies

\begin{lemma}\label{lem:filtered-basis}
Let $C$, $A_1$ and $A_2$ be filtered algebras, and $\lambda_1: C \to A_1$ and $\lambda_2: C \to A_2$ be filtered free morphisms, with filtered basis $\mathcal{B}_i$, with $1 \in \mathcal{B}_i$.
By Proposition~\ref{prop:normalform}, the image of $\Psi_\circ$ forms a basis for $A_1\ast_{C} A_2$. Then we can endow $B=A_1\ast_{C} A_2$ with the algebra filtration induced by
 \begin{align*}
         \nu(\psi_{i_1}(v_1)\dots \psi_{i_n}(v_n)) \coloneqq \nu_{A_{i_1}}(v_1)+\dots+\nu_{A_{i_n}}(v_n).
    \end{align*}
    so that $A_1\ast_{C} A_2$ is filtered free as $C$-module (with the $C$-module structure given by the composition $\psi_1 \circ \lambda_1=\psi_2 \circ \lambda_2$), with filtered basis formed by the image of $\Psi_\circ$. Moreover, if $C$, $A_1$ and $A_2$ are separated, then also $B$ is separated.
\end{lemma}
\begin{proof} The lemma is a direct consequence  from the previous considerations. 
In order to check that $B$ is separated if $C$, $A_1$ and $A_2$ are, note that this implies that the values $\nu_{A_{i_1}}(v_1)+\dots+\nu_{A_{i_n}}(v_n)$ are always integers and as $C$ is also separated one concludes that also $A_1\ast_{C} A_2$ is. 
\end{proof}

We call $\nu$ the \emph{standard filtration} on $A_1\ast_{C} A_2$.


\begin{lemma}\label{lem:normalform-filtered-strict}
    With the standard filtration, $\psi_1$ and $\psi_2$ are strict
\end{lemma}

\begin{proof}
    For $i=1,2$,  take an element $a \in A_i$, with $a=\sum_{v \in \mathcal{B}_1}vc_{v}$. Then
    \begin{align*}
        \nu(\psi_i(a))
        &=\nu\p{\sum_{v \in \mathcal{B}_i}\psi_i(v)\psi_i(\lambda_i(c_{v}))}
        =\nu\p{\sum_{v \in \mathcal{B}_i}\Psi(v)c_v} \\
        &=\inf(\nu_{A_i}(v)+\nu_C(c_v))
        =\nu_{A_i}(a).
    \end{align*}
\end{proof}

\subsection{Augmented algebras}

Augmented algebras are an important particular case of filtered algebras.

\begin{definition}\label{def:augmented} An $\F$-algebra $A$ together with an algebra epimorphism $\eps:A\to F$ is said to be an \emph{augmented $\F$-algebra}. In that case $A^+=\ker\eps$ is called the \emph{augmentation ideal} and $A$ has a filtration given by $F^nA=(A^+)^n$ that we call the \emph{augmentation filtration}.
\end{definition}

For example, if $G$ is a group, then the group algebra $\F G$ is augmented, where
$\eps\colon\F G\to\F$ satisfies $\epsilon(g)=1$ for all $g\in G$. In this case we denote the augmentation ideal by $\omega_G$.

We consider now free amalgamated products of augmented algebras. Assume that $A_1$, $A_2$ and $C$ are augmented algebras endowed with  the augmentation filtrations and
$\lambda_1: C \to A_1$ and $\lambda_2: C \to A_2$ are augmented maps, i.e., maps such that $\lambda_i(C^+)\subseteq A_i^+$ for $i=1,2$. Then the universal property of push-outs implies that the free amalgamated product $B=A_1\ast_C A_2$ is also augmented and that the maps $\psi_i:A_i\to B$ are augmented.

\begin{proposition}\label{prop:standardaugmentation} Let $A_1$, $A_2$ and $C$ be augmented algebras endowed with the augmentation filtrations and let
$\lambda_1: C \to A_1$ and $\lambda_2: C \to A_2$ be augmented maps. Let $B=A_1\ast_C A_2$. If
    \begin{itemize}
       \item[\textup{(i)}] $\lambda_1$ and $\lambda_2$ are filtered free,
        \item[\textup{(ii)}] $\psi_1(A_1^+)$ and $\psi_2(A^+_2)$ generate $B^+$,
    \end{itemize}
    then the augmentation filtration on $B$ coincides with the standard filtration.
\end{proposition}
\begin{proof} First note that by (i) and Lemma \ref{lem:filtered-basis} one is able to define the standard filtration on $B$. Denote by $\nu'$ the augmentation filtration of $B$ and by $\paren{F^nB}_{n\geq 0}$ the ideals induced by the standard filtration $\nu$, so that, for any $b \in B$
    \begin{itemize}
        \item $\nu(b)\geq n$ $\iff$ $b \in F^nB$, and
        \item $\nu'(b)\geq n$ $ \iff$ $b \in (B^+)^n$.
    \end{itemize} 
    By construction of $\nu$, $\psi_1$ and $\psi_2$ are filtered with respect to $\nu$. So both $\psi_1(A_1^+)$ and $\psi_2(A_2^+)$ are contained in $F^1B$  hence (ii) implies $B^+\subseteq F^1B$, and
    \begin{align*}
        (B^+)^n\subseteq (F^1B)^n\subseteq F^nB,
    \end{align*}
    from which it follows that $\nu(b)\geq \nu'(b)$ for any $b \in B$. 
    
    On the other hand, the maps $\psi_1$ and $\psi_2$ are augmented, so they are filtered with respect to $\nu'$. This implies $\nu_{A_i}(v)\leq\nu'(\psi_i(v))$ for any $v\in A_i$  and $i=1,2$. Therefore    
        \begin{align*}
        \nu'(\psi_{i_1}(v_1)\dots \psi_{i_n}(v_n)) &\geq \nu'(\psi_{i_1}(v_1))) + \dots +\nu'(\psi_{i_n}(v_n))) \\
        &\geq \nu_{A_{i_1}}(v_1)+\dots+\nu_{A_{i_n}}(v_n)=\nu(\psi_{i_1}(v_1)\dots \psi_{i_n}(v_n)).\\
    \end{align*}
    so $\nu=\nu'$.\end{proof}

\subsection{Graded algebras}

\begin{definition}
An $\F$-algebra $A$ is \emph{$\Z_{\geq 0}$-graded} or simply \emph{graded} if it admits a vector space decomposition
$$A=\oplus_{n\geq 0}(A^n)$$
such that $A^nA^m\subseteq A^{n+m}$. If $A$, $B$ are graded, a morphism $\varphi:A\to B$ is \emph{graded} if $\varphi(A^n)\subseteq B^n$.
If $A$ is a graded algebra, a \emph{graded $A$-module} is an $A$-module $M$ that admits an $\F$-vector space decomposition
$$M=\oplus_{n\geq 0}(M^n)$$
so that $a\in A^n$, $x\in M^m$ implies $m\cdot a\in M^{n+m}$.  Elements in a graded algebra (or module) lying in a single component $A^n$ (or $M^n$) are said to be 
\emph{homogeneous}.
\end{definition}

Analogously as we did for filtered algebras, we will need a notion of graded free modules (and maps).

\begin{definition} If $C$ is a graded algebra, a graded $C$-module $M$ is 
\emph{graded free} if it admits a basis consisting of homogeneous elements. A graded morphism $\varphi:C\to A$ is said to be \emph{graded free} if $A$ is a graded free  $C$-module. 
\end{definition}

Coming back to free amalgamated products of algebras, if  $C$, $A_1$ and $A_2$ are graded algebras, and $\lambda_1: C \to A_1$ and $\lambda_2: C \to A_2$ are graded free morphisms, with homogeneous basis $\mathcal{B}_i$, with $1 \in \mathcal{B}_i$, then 
by Proposition \ref{prop:normalform} the image of $\Psi$ forms a basis for $A_1\ast_{C} A_2$. We can endow it with a grading
\begin{align*}
    \deg(\psi_{i_1}(v_1)\dots \psi_{i_n}(v_n)) \coloneqq \deg_{A_{i_1}}(v_1)+\dots+\deg_{A_{i_n}}(v_n).
\end{align*}
so that $\psi_1 \circ \lambda_1=\psi_2 \circ \lambda_2 : C \to A_1\ast_{C} A_2$ is graded free, with homogeneous basis formed by the image of $\Psi$.

\medskip

Next, we want to relate filtered and graded algebras. The standard way to do that is using the functor $\gr$.

\begin{definition}
    Let $A$ be a filtered algebra with filtration $F^nA$, $n\geq 0$. Then 
    $$\gr_\bullet A=\oplus_{n\geq 0}\gr_nA$$ 
    with
    $\gr_nA= F^nA/F^{n+1}A$
    is a graded algebra. For each $a\in A$ with $\nu_A(a)=n$,  we denote by $\overline{a}$ the corresponding element in $\gr_n(A)\subseteq \gr_\bullet A$. Similarly we define 
    $\gr_\bullet M$ for a filtered $A$-module $M$.
        If $B$ is also filtered and $\phi: A \to B$ is a filtered morphism, we denote by $\gr_\bullet \phi: \gr_\bullet A \to \gr_\bullet B$ the graded morphism induced by $\phi$.
\end{definition}

Observe that  $\overline{a}=0$ if and only if $a\in\cap F^nA$. The following Lemma allows us to relate filtered and graded bases.

\begin{lemma}\cite[Lemma 2.3.12]{Lazard2}\label{lemma_bases}
    Let $C$ be a filtered algebra, let $M$ be a filtered $C$-module and let $X$ be a set of elements of $M$. Assume that both $C$ and $M$ are separated. 
    Then, if $X$ is a filtered $C$-basis for $M$,  $\overline{X}$ is a homogeneous $\gr_\bullet C$-basis for $\gr_\bullet M$.
    
    Conversely, if $\overline{X}$ is a homogeneous $\gr_\bullet C$-basis for $\gr_\bullet M$, then $X$ is filtered-free.

    As a consequence, if a morphism $\phi: C \to A$ of filtered algebras  is filtered free, then $\gr_\bullet \phi: \gr_\bullet C \to \gr_\bullet A$ is graded free.
\end{lemma}
\begin{proof}
    Suppose that $X=(x_i)_{i\in I}$ is a filtered $C$-basis for $M$.
     Let $(a_i)_{i \in I}$ be a family of elements of $C$ where almost all $a_i$'s are zero, such that $\sum_i \overline{x_i}\cdot\overline{a_i}=0$.
We may assume that $\overline{x_i}\cdot\overline{a_i}$ are homogeneous elements of the same degree $k$. Then $\sum_i \overline{x_i}\cdot\overline{a_i}=0 $ is equivalent to $\sum_ix_ia_i \in F^{k+1}M$ and this is equivalent to
$$\nu_M\p{\sum_ix_ia_i} \geq k+1.$$
But since $(x_i)_{i\in I}\subseteq X$ and $X$ is a filtered basis, for every $i \in I$
\begin{align*}
    \nu_M\p{\sum_ix_ia_i}=\nu_M(x_i)+\nu_C(a_i)=\deg(\overline{x_i})+\deg(\overline{a_i})=\deg{\overline{x_i}\cdot\overline{a_i}}=k,
\end{align*}
contradicting the hypothesis.
Conversely, suppose that $(\overline{x_i})_{i\in I}$ is a homogeneous $\gr_\bullet C$-basis for $\gr_\bullet M$. Let $(a_i)_{i \in I}$ be a family of elements of $C$ where almost all $a_i$'s are zero, such that
\begin{align*}
    \nu_M\p{\sum_ix_ia_i} > \min_i\paren{\nu_M(x_i)+\nu_C(a_i)} \eqqcolon k.
\end{align*}
Let $J\subseteq I$ be the subset of indexes $i$ such that $\nu_M(x_i)+\nu_C(a_i)=k$, i.e. the set of indexes for which the elements $\overline{x_i}\cdot\overline{a_i}$ are homogeneous of degree $k$. Then the component of $\sum_{i}\overline{x_i}\cdot\overline{a_i}$  of degree $k$ is
\begin{align*}
    \sum_{i\in J}\overline{x_i}\cdot\overline{a_i}=\sum_{i \in J}x_ia_i+F^{k+1}M=0,
\end{align*}
which contradicts the hypothesis.
\end{proof}

Using Lemma \ref{lemma_bases} one is able to compare the free amalgamated product of filtered algebras with the free amalgamated product of the graded versions. Observe first that
 if $C$, $A_1$ and $A_2$ are filtered algebras, and $\lambda_1: C \to A_1$, $\lambda_2: C \to A_2$ are filtered free morphisms, the universal property of push-outs implies that there is a natural map 
 $$\overline{\tau}_\bullet: \gr_\bullet A_1 \ast_{\gr_\bullet C} \gr_\bullet A_2 \to \gr_\bullet B.$$ 

\begin{proposition}\label{prop:filtered-graded}
    Let $C$, $A_1$ and $A_2$ be filtered separated algebras, and let $\lambda_1: C \to A_1$ and $\lambda_2: C \to A_2$ be filtered free morphisms. If $B\coloneqq A_1 \ast_C A_2$ is filtered with the standard filtration, then the natural map $\overline{\tau}_\bullet: \gr_\bullet A_1 \ast_{\gr_\bullet C} \gr_\bullet A_2 \to \gr_\bullet B$ is an isomorphism of graded algebras.
\end{proposition}

\begin{proof} Consider, for $i=1,2$,  a filtered $C$-basis  $\mathcal{B}_i$ for $A_i$. Then, by Lemma~\ref{lem:filtered-basis}, one has
    a filtered $C$-basis $\mathcal{V}$ for $B$ with respect to the standard filtration. (Recall that the $C$-module structure of $B$ is given by the maps $\psi_1 \circ \lambda_1=\psi_2 \circ \lambda_2$).  Then, Lemma \ref{lemma_bases} implies that $\overline{\mathcal{V}}$ is a graded $\gr_\bullet C$-basis of 
    $\gr_\bullet B$.

    On the other hand, Lemma \ref{lemma_bases} also implies that $\overline{\mathcal{B}_i}$ is a graded $\gr_\bullet C$-basis of 
    $\gr_\bullet A_i$ and  Proposition \ref{prop:normalform} produces a graded 
    $\gr_\bullet C$-basis $\mathcal{W}$ of  
    $\gr_\bullet A_1 \ast_{\gr_\bullet C} \gr_\bullet A_2$. 
    
    Clearly, $\overline{\mathcal{V}}$ is precisely the image of $\mathcal{W}$ under the natural map $\overline{\tau}_\bullet: 
    \gr_\bullet A_1 \ast_{\gr_\bullet C} \gr_\bullet A_2 \to \gr_\bullet B$. Since $\overline{\tau}_\bullet$ induces a bijection between basis, it is an isomorphism.
\end{proof}


\section{Descending central series}\label{sec:dcs}
If $G$ is a group and $R$ is a commutative unital ring, the group ring $R G$ is augmented with augmentation map $\eps\colon R G\to R$ induced by the map $g\mapsto 1$, $g\in G$.  Therefore, $R G$ has a natural filtration by powers of the augmentation ideal
$\omega^R_G=\ker\epsilon.$

A natural question is whether one can relate this filtration with a filtration of the defining group. It is easy to see that $\omega^R_G$ is precisely the ideal of $R G$ generated by the elements $g-1$ where $g\in G$ and this leads to the definition of the \emph{dimension subgroups}
$$D_n(G,R)=\{g\in G\mid g-1\in(\omega^R_G)^n\}.$$
One can check that these are indeed subgroups of $G$ and clearly
$$G=D_1(G,R)\geq D_2(G,R)\geq\ldots\geq D_n(G,R)\geq D_{n+1}(G,R)\geq\ldots$$
In the case when we work with coefficients in $\Z$, these subgroups are extremely difficult to determine in group theoretical terms. However, working with coefficients in a field $\F$ the situation changes and there are nice description available in terms of  descending central series.

\begin{definition} 
\label{def:gamma}
Let $G$ be a group. The \emph{descending central series} of $G$ is the series given by $\gamma_1(G)=G$ and
$$\gamma_n(G)=[G,\gamma_{n-1}(G)].$$
If $p$ is either zero or a prime, we also define following \cite{BHS}, $p$-local versions by
$$\begin{aligned}
\gamma_n^{[0]}(G)&=\{x\in G\mid x^m\in\gamma_n(G)\text{ for some }m>0\},\\
\gamma_n^{[p]}(G)&=\prod_{jp^s\geq n}\gamma_j(G)^{p^s}\text{ if $p$ is a prime}.
\end{aligned}$$
In the prime case, the series $\gamma_n^{[p]}(G)$ is called  the 
\emph{$p$-Zassenhaus central series}, and there is another well known $p$-local central series, the \emph{$p$-Stallings series}, defined by
$$\gamma_{n,p}(G)=\prod_{j+s\geq n}\gamma_j(G)^{p^s}\text{ if $p$ is a prime}.$$
\end{definition}

Then one has
\begin{theorem}[Quillen-Jennings-Lazard] Let $p$ be zero or a prime and $\F$ a field of characteristic $p$. Then
$$\gamma_n^{[p]}(G)=D_n(G,\F).$$
\end{theorem}

{ We will say that the group $G$ is \emph{residually $[p]$} if 
$$\cap_n\gamma_n^{[p]}(G)=1.$$
If $p$ is a prime, this is equivalent to the following: 
for every element $1\neq g\in G$ there is a normal subgroup $N\trianglelefteq G$ such that $g\not\in N$ and $G/N$ is a nilpotent $p$-group of bounded exponent. If moreover $G$ is finitely generated, this means that $G$ is residually $p$-finite. If $p=0$, one verifies easily that all the quotient groups   $G/\gamma_n^{[0]}(G)$ are torsion free nilpotent group and deduce that being residually $[0]$ is the same as saying that for each element $1\neq g\in G$ there is a normal subgroup $N\trianglelefteq G$ with $G/N$ torsion free nilpotent such that $g\not\in N$.
These groups are usually called \emph{residually torsion free nilpotent}, but we keep the name \emph{residually $[0]$} to be able to treat at the same time the cases  $p=0$ and $p$  a prime. 
Being residually $[p]$ is the group theoretical analog of  separability for the group algebra because of the following consequence of a Theorem by Lichtman (cf. \cite[Theorem 2.1]{Lichtman2}):

\begin{proposition}
\label{prop:resp}
Let  $p$ be either zero or a prime and $\F$ a field of characteristic $p$. A group $G$ is residually $[p]$ if, and only if, $\F G$ is separated with respect to the augmentation filtration. Moreover, in this case $G$ is residually $p$-finite provided $G/\gamma_2^{[p]}(G)$ is finite.
\end{proposition} 
\begin{proof} Assume that $\F G$ is separated with respect to the augmentation filtration, i.e., $\cap_{n\geq 0}\omega_G^n=0$. Let $g\in\cap_{n\geq 1}\gamma_n^{[p]}(G)$. As $\gamma_n^{[p]}(G)=D_n(G,\F)$, then $g-1\in\cap_{n\geq 0}\omega_G^n=0$. 

For the converse, let $\omega_G^\Z$ be the augmentation ideal of $\Z G$. We assume that $G$ is residually $[p]$, this means that $G$ is either  residually torsion free nilpotent (if $p=0$) or that for the prime $p$, $G$ is residually a nilpotent $p$-group of bounded exponent. Then Theorem 2.1 of \cite{Lichtman2} implies that $\cap_{n\geq 0}(\omega_G^\Z)^n=0$ and tensoring with the field $\F$ we obtain the result.
\end{proof}
}

The series $\gamma_n^{[p]}(G)$ is an $N$-series in Lazards' terminology, meaning that 
$$[\gamma^{[p]}_n(G),\gamma^{[p]}_m(G)]\leq \gamma^{[p]}_{n+m}(G).$$
This has the consequence that one can use the group commutator to define a Lie algebra

$$\F\gr^{[p]}_\bullet(G)=\bigoplus_{i\geq 1}\gamma_i^{[0]}(G)/\gamma_{i+1}^{[0]}(G)\otimes_\Z\F.
$$
To simplify notation, we will just write $\gr^{[p]}_\bullet(G)$ if the coefficient field is either $\Q$ or the field of $p$ elements $\F_p$.
In the case when $p$ is a prime, the algebra $\gr^{[p]}_\bullet(G)$ has additional structure: it is a $p$-restricted Lie algebra, meaning that there is a map  
$a\mapsto a^{[p]}$ such that for elements $a,b$ and $t\in\F$
\begin{itemize}
\item[(i)] $[a^{[p]},b]=[a,\buildrel{p}\over\ldots,a,b]$,

\item[(ii)] $(ta)^{[p]}=t^pa^{[p]}$,

\item[(iii)] $(a+b)^{[p]}=a^{[p]}+b^{[p]}+\sum_{i=1}^{p-1}s_i(a,b)$ where $s_i(a,b)$ is the coefficient of $t$ in $[ta+b,\buildrel{p-1}\over\ldots,ta+b,a]$.
\end{itemize}
This algebra was first introduced by Zassenhaus in 1939 so it is often called the 
\emph{restricted $p$-Zassenhaus Lie algebra} of the group $G$.

The $p$-Stallings central series is also an $N$-series and one can define a Lie algebra using the succesive quotients as we did for the $p$-Zassenhaus series. However, that Lie algebra has not necessarily the structure of a $p$-restricted Lie algebra in general.
Later we will make use of the following relation between the $p$-Zassenhaus and the $p$-Stallings central series:

\begin{lemma}\label{lem:ZassenhausStallings} Let $p$ be a prime. Then
$$\gamma_n^{[p]}(G)=\prod_{jp^s\geq n}\gamma_{j,p}(G)^{p^s}.$$
\end{lemma}
\begin{proof} Taking into account that for $i\leq j$ one has  $\gamma_r(G)^{p^i}\geq\gamma_r(G)^{p^j}$,  we conclude that
by definition of $\gamma_n^{[p]}(G)$ we have
\begin{equation}
    \label{eq:Jen}
\gamma_n^{[p]}(G)=\prod_{1\leq r\leq n}\gamma_r(G)^{p^{i_r}}
\end{equation}
where $i_r=\lceil\log_p{\big
(\frac{n}{r}\big)}\rceil$, i.e., $i_r$ is the smallest integer such that 
$$p^{i_r}\geq \frac{n}{r}.$$
Similarly, if we show
\begin{equation}\label{eq:ZS}
\gamma_n^{[p]}(G)=\prod_{1\leq r\leq n}\gamma_{r,p}(G)^{p^{i_r}}
\end{equation}
we will get the desired result. In the left hand side of (\ref{eq:ZS}), the factor $\gamma_r(G)$ 
appears with exponent $p^{i_r}$ and in the right hand side one has
$$\prod_{r\leq j\leq n}(\gamma_r(G)^{n-j})^{p^{i_{n-j+r}}}.$$
Hence in order to prove (\ref{eq:ZS}) one needs to verify that for $1\leq r\leq j\leq n$ one has $n-j+i_{n-j+r}\geq i_r$, or equivalently, that
\begin{equation}\label{eq:ZS2}
p^{n-j+i_{n-j+r}}\geq\frac{n}{r}.
\end{equation}
As
$$p^{n-j+i_{n-j+r}}=p^{n-j}p^{i_{n-j+r}}\geq 2^{n-j}\frac{n}{n-j+r}$$
it is enough to check that for $1\leq r\leq j\leq n$
$$2^{n-j}r\geq n-j+r$$
and this follows easily by induction.
\end{proof}

Quillen \cite[Theorem 1]{Quillen} has shown that there is an isomorphism 
which we will call the \emph{Quillen map}
\begin{equation}\label{eq:Quillen}Q_G:\ur(\gr^{[p]}_\bullet(G))\to\gr_\bullet \F G.\end{equation}
Here 
$$\gr_\bullet \F G=\oplus \omega_G^i/\omega^{i+1}_G,$$
is the graded algebra obtained from the augmentation filtration of $\F G$, and $\ur$ denotes the universal enveloping algebra functor. This isomorphism is induced by the map 
$$g\gamma_i(G)\mapsto g-1+\omega_G^i$$
for $g\in\gamma_{i-1}(G)$.


\begin{definition} Let $p$ be either zero or a prime and let $H\leq G$ be a subgroup of the group $G$. We say that $H$ is \emph{strictly $p$-embedded} in $G$ if for all $i\geq 1$
$$H\cap\gamma_i^{[p]}(G)=\gamma_i^{[p]}(H).$$
\end{definition}

The most obvious examples of strictly embedded subgroups are retracts: recall that $H$ is a retract of a group $G$ if $H\leq G$ and there is an epimorphism $\pi:G\to H$ that restricts to the identity in $H$. Then, taking the kernel $K$ of $\pi$, we see that $G$ decomposes as a semidirect product $G=H\ltimes K$. 
It is well known (see for example \cite[Lemma 1.5]{AF}) that retracts behave well with respect to central filtrations and as a consequence are strictly p-embedded for any $p$. In fact, for any $p$ and any $n\geq 1$
\begin{equation}\label{eq:retract}\gamma_n^{[p]}(G)=\gamma_n^{[p]}(H)(\gamma_n^{[p]}(G)\cap K).\end{equation}

\begin{proposition}\label{prop:strictly-embedded}
    Let $p$ be either zero or a prime, $\F$ a field of characteristic $p$ and $H$ a subgroup of the group $G$. Assume that $G$ is residually $[p]$. Then the following are equivalent.
    \begin{itemize}
        \item[\textup{(i)}] $H$ is strictly  $p$-embedded in $G$.
        \item[\textup{(ii)}] The inclusion map $\F[H]\to \F[G]$ is filtered free (for the augmentation filtration).
    \end{itemize}
\end{proposition}
\begin{proof} Condition  (i) is equivalent to the induced map $\gr_\bullet^{[p]}(H) \to\gr_\bullet^{[p]}(G)$ being injective.
   By the Poincaré-Birkhoff-Witt Theorem and its $p$-restricted version (cf. \cite[Theorem 2.5.5.1]{Strade}) this is equivalent to  the map 
   $$\ur(\gr_\bullet^{[p]}(H)) \to\ur(\gr_\bullet^{[p]}(G))$$
   being graded free. Then using Quillen's Theorem (cf. \cite[Theorem 1]{Quillen}), we see that this happens if and only if  the map $\gr_\bullet(\F[H]) \to \gr_\bullet (\F[G])$ is graded free.
   Now, by Lemma \ref{lemma_bases} this is in turn equivalent to (ii).
\end{proof}

\section{Proof of Theorems \ref{thm:graded} and \ref{teo:strictly}}\label{sec:proofAB}

We begin with the case of a free amalgamated product, which is the statement of the next Theorem.

\begin{theorem}\label{thm:amalgamated} Let $p$ be either zero or a prime and $G_1$ and $G_2$  groups with a common subgroup $H$. Assume that  $H$ is  $p$-strictly embedded in $G_1$ and $G_2$. Then
$$\gr^{[p]}_\bullet(G)=\gr^{[p]}_\bullet(G_1)\ast_{\gr^{[p]}_\bullet(H)}\gr^{[p]}_\bullet(G_2)$$
where $G=G_1\ast_HG_2$.
\end{theorem}
\begin{proof}  We first show the theorem  in the case when $G_1$ and $G_2$ are residually $[p]$. Observe that as $H$ is strictly  $p$-embedded in both, this will imply that also $H$ is  residually $[p]$. Consider the functor $F$ from the category of groups to the category of rings given by $G\to \F G$. The functor sending a ring to its group of units is right adjoint to $F$. The fact that $F$ has a right adjoint implies that it preserves push-outs (cf. \cite[Chapter V Section 5]{MacLane}), and therefore for the free amalgamated product $G=G_1\ast_HG_2$ one has
$$\F G=\F G_1\ast_{\F H}\F G_2$$
 as associative algebras. We equip $\F G_1$, $\F G_2$ and $\F H$ with the augmentation filtrations. 
 By Proposition~\ref{prop:strictly-embedded} and the hypothesis on $H$, the maps $\F H\to \F G_1$ and $\F H\to \F G_2$ are filtered free. By the 
 Lemmata~\ref{lem:filtered-basis} and \ref{lem:normalform-filtered-strict} the augmentation filtrations of $\F G_1$ and $\F G_2$ induce a filtration (the standard filtration) on $\F G$ so that the maps $\F G_1\to\F G$ and $\F G_2\to \F G$ are strict. We claim that this induced filtration is precisely the augmentation filtration. To prove the claim, we only have to check the hypothesis of 
 Proposition~\ref{prop:standardaugmentation}, note that the augmentation map induced on $\F G$ by the push out and the augmentations of $\F G_1$ and $\F G_2$ is precisely the natural augmentation map of $\F G$. The corresponding augmentation ideal $\omega_G$ is generated by elements of the form $1-g$ for $g\in G$.
For any element $g \in G$, by the normal form theorem, $g=g_0g_1\dots g_k$ where $g_0,g_1,\ldots,g_k$ lie in $G_1$ or $G_2$ (alternatively). Therefore
\begin{align*}
    (1-g)=(1-g_{0})+(1-g_{1}\dots g_{k})-(1-g_{0})(1-g_{1}\dots g_{k})
\end{align*}
so arguing by induction on $k$ we deduce that  $\omega_G$ is generated by $\omega_{G_1}$ and $\omega_{G_2}$.
Then, by applying Proposition \ref{prop:filtered-graded}, we establish that the natural map of graded algebras
\begin{align*}
    \gr_\bullet(\F[G_1]) \ast_{\gr_\bullet(\F[H])}\gr_\bullet(\F[G_2]) \to 
    \gr_\bullet(\F[G])
\end{align*}
is an isomorphism. Using Quillen map we see that also
\begin{align*}
    \ur(\gr^{[p]}_\bullet(G_1)) \ast_{\ur(\gr^{[p]}_\bullet(H))}\ur(\gr^{[p]}_\bullet(G_2)) \to \ur(\gr^{[p]}_\bullet(G))
\end{align*}
is an isomorphism. Finally, taking into account that both the functor that sends a Lie algebra in zero characteristic to its universal enveloping algebra and the functor that sends a $p$-restricted Lie algebra where $p$ is a prime to its restricted universal enveloping are fully faithful, we get the same (but) for the Lie algebras.


Consider now the general case. 
 Let $K_i=\cap_n\gamma_n^{[p]}(G_i)$ for $i=1,2$. Then
$\gr^{[p]}_\bullet(G_i)=\gr^{[p]}_\bullet(G_i/K_i)$ and for $\cap_n\gamma_n^{[p]}(H)=H\cap K_1=H\cap K_2$, also $\gr^{[p]}_\bullet(H)=\gr^{[p]}_\bullet(H/H\cap K_1)$. Moreover, $G_1/K_1$ and $G_2/K_2$ are residually $[p]$. 
Let $K$ be the normal subgroup of $G$ generated by $K_1$ and $K_2$. Then $G/K=(G_1/K_1)\ast_{H/H\cap K_1}(G_2/K_2)$. For each $n\geq 1$, $K$ lies inside the normal subgroup of $G$ generated by $\gamma_n^{[p]}(G_1)$ and $\gamma_n^{[p]}(G_2)$ and as both lie inside $\gamma_n^{[p]}(G)$ we get  $K\leq \gamma_n^{[p]}(G)$ for any $n$ and therefore
$$\gr^{[p]}_\bullet(G)=\gr^{[p]}_\bullet(G/K).$$
So the residually $[p]$ case implies
$$gr^{[p]}_\bullet(G)=\gr^{[p]}_\bullet(G/K)=\gr^{[p]}_\bullet(G_1/K_1)\ast_{\gr^{[p]}_\bullet(H/H\cap K_1)}\gr^{[p]}_\bullet(G_2/K_2)$$
and from this we get the same for $G$ with respect to $G_1$, $G_2$ and $H$.
 \end{proof}

To prove the second part of Theorem \ref{thm:graded} we first need  a $p$-version of a theorem by Falk and Randell that we will obtain easily using a result of Bellingeri and Gervais. Let $p$ be either zero or a prime and let $G=H\ltimes K$ be a semi-direct product. We say that $G$ is an \emph{almost $p$-direct product} if the induced action of $H$ on $K/\gamma_2^{[p]}(K)$ is trivial. In the case when $p=0$, this is usually called just an \emph{almost direct product}. Consider now the Lie algebras $\gr^{[p]}_\bullet(H)$ and $\gr^{[p]}_\bullet(K)$ which are generated by their  components of degree one. Take $\mathfrak{a}\in\gr^{[p]}_1(H)$ and $\mathfrak{b}\in\gr^{[p]}_1(K)$ and choose $h\in H$, $k\in K$ such that $\mathfrak{a}=h\gamma_2^{[p]}(H)$ and $\mathfrak{b}=k\gamma_2^{[p]}(K)$. Then
$[k,h]=k^{-1}k^h\in\gamma_2^{[p]}(K)$ so
$$\mathfrak{b}\mapsto[\mathfrak{a},\mathfrak{b}]=[k,h]\gamma_3^{[p]}\in\gr^{[p]}_\bullet(K)$$
defines a derivation $d$ on $\gr^{[p]}_\bullet(K)$. Using this derivation we may define the  semidirect product of Lie algebras
$$\gr^{[p]}_\bullet(H)\ltimes_d\gr^{[p]}_\bullet(K).$$

\begin{lemma}\label{lem:FRp} Let $G=H\ltimes K$ be 
an almost $p$-direct product. Then for each $n$, $\gamma_n^{[p]}(G)=\gamma_n^{[p]}(H)\ltimes\gamma_n^{[p]}(K)$ and, as a consequence, both $H$ and $K$ are strictly $p$-embedded in $G$ and
$$\gr^{[p]}_\bullet(G)=\gr^{[p]}_\bullet(H)\ltimes_d\gr^{[p]}_\bullet(K).$$
\end{lemma}
\begin{proof} In the case when $p=0$, this is just \cite[Theorem 3.1]{FalkRandell}. Assume that $p$ is a prime. Then the fact that $G$ is an almost $p$-direct product implies by 
\cite[Theorem 4]{BG} that for any $n\geq 1$, $\gamma_{n,p}(G)=\gamma_{n,p}(H)\ltimes\gamma_{n,p}(K)$. From this and Lemma
\ref{lem:ZassenhausStallings} we deduce that also $\gamma_n^{[p]}(G)=\gamma_n^{[p]}(H)\ltimes\gamma_n^{[p]}(K)$. 
Taking quotients this implies that $\gr^{[p]}_\bullet(K)$ is precisely the kernel of the epimorphism
$$\gr^{[p]}_\bullet(G)\to\gr^{[p]}_\bullet(H)$$
induced by the projection $G\to H$. Moreover, if we denote by $L_\bullet$ the semidirect product of Lie algebras 
$$L_\bullet=\gr^{[p]}_\bullet(H)\ltimes_d\gr^{[p]}_\bullet(K)$$
then from the relators in $G$ we see that the relations associated to the derivation $d$ hold true in $\gr^{[p]}_\bullet(G)$ so we have a well defined epimorphism $\tau:L_\bullet\to\gr^{[p]}_\bullet(G)$ that fits in a commutative diagram
$$\begin{tikzcd}
&\gr^{[p]}_\bullet(K)\arrow[r,hook]\arrow[d,"\tau_1"]&L_\bullet\arrow[r,two heads]\arrow[d,"\tau"]&\gr^{[p]}_\bullet(H)\arrow[d,"\tau_2"]\\
&\gr^{[p]}_\bullet(K)\arrow[r,hook]&\gr^{[p]}_\bullet(G)\arrow[r,two heads]&\gr^{[p]}_\bullet(H)\\
\end{tikzcd}$$

and as the maps $\tau_1$ and $\tau_2$ are isomorphisms, so is $\tau$.
\end{proof}

We extend the notion of almost $p$-direct products to HNN-entensions: Let $H\leq G_1$ be a subgroup and consider a map $\varphi:H\to H$ which is trivial on $H/\gamma_2^{[p]}$. Then we say that the HNN extension
$$G=G_1\ast_{H,\varphi}$$
is \emph{almost $p$-trivial}.
We are  now able to prove the second part of Theorem \ref{thm:graded}.

\begin{corollary}\label{cor:HNN} Let $H\leq G_1$ be a strictly $p$-embedded  subgroup and $\varphi:H\to H$ an automorphism that acts trivially on $H/\gamma^{[p]}_2(H)$. Let $$G=G_1\ast_{H,\varphi}.$$
Then
$$\gr_\bullet^{[p]}(G)=\gr_\bullet^{[p]}(G_1)\ast_{\gr_\bullet^{[p]}(H),d}$$
where $d$ is the derivation of $\gr_\bullet^{[p]}(H)$ induced by $\varphi$.
\end{corollary}

\begin{proof} Let $G_2=\langle t\rangle\ltimes H$, where $t$ is the stable letter and the $t$ action is given by $\varphi$. Then $G_2$ is an almost $p$-direct product. So $H$ is strictly embedded in $G_2$ by Lemma~\ref{lem:FRp}. Moreover
$$G=G_1\ast_HG_2$$
and as - by hypothesis - $H$ is strictly $p$-embedded in $G_1$, Theorem \ref{thm:amalgamated} implies that 
$$\gr_\bullet^{[p]}(G)=\gr_\bullet^{[p]}(G_1)\ast_{\gr_\bullet^{[p]}(H)}\gr_\bullet^{[p]}(G_2).$$
From this the obvious modification of the presentation of $\gr_\bullet^{[p]}(G)$ gives the desired result.
\end{proof}

The proof of Theorem \ref{thm:amalgamated} also implies Theorem \ref{teo:strictly}.

\medskip

\noindent\emph{Proof of Theorem \ref{teo:strictly}.} Consider first the case of a free amalgamated product $G=G_1\ast_HG_2$.
Assume first that $G_1$ and $G_2$ are residually $[p]$. The proof of Theorem \ref{thm:amalgamated} implies that
 $\F G$ is separated and that the embeddings $\F G_1\to\F G$ and $\F G_2\to\F G$ are strict for the augmentation filtrations. As $\F G$ is separated, $G$ is residually $[p]$. Since the embeddings $\F G_1\to\F G$ and $\F G_2\to\F G$ are strict, one deduces that $G_1$ and $G_2$ are strictly $p$-embedded in $G_1\ast_HG_2$. To see this, consider for example $g\in G_1\cap\gamma_n^{[p]}(G)$. Then 
$$1-g\in\F G_1\cap\omega_G^n=\omega_{G_1}^n$$
 so $g\in\gamma_n^{[p]}(G_1)$. 

 Next, we  have to show that $G_1$ and $G_2$ are strictly $p$-embedded in $G$ even if they are not residually $[p]$. Let  $K_i=\cap_n\gamma_n^{[p]}(G_i)$ for $i=1,2$ and $K$ the normal subgroup of $G$ generated by $K_1$ and $K_2$. As in the proof of Theorem \ref{thm:amalgamated}, we have $G/K=(G_1/K_1)\ast_{H/H\cap K_1}(G_2/K_2)$ and $G_1/K_1,G_2/K_2$ are residually $[p]$ so by the previous case, $G_1/K_1$ and $G/2/K_2$ are strictly $p$-embedded in $G/K$. From this and using the fact that for any $n$, $K_i\leq\gamma_n^{[p]}(G_i)$ for $i=1,2$ and  $K\leq\gamma_n^{[p]}(G)$ one easily deduces that $G_1$ and $G_2$ are also strictly $p$-embedded in $G$.
 
 For HNN extensions, one may argue as in Corollary \ref{cor:HNN}, taking into account that Lemma~\ref{lem:FRp} implies that if the HNN extension has edge group $H$ and stable letter $t$, the almost $p$-direct product $G_2=\langle t\rangle\ltimes H$ is residually $[p]$ if, and only if, $H$ is. 
 \qed

\subsection{Profinite groups}
\label{ss:prop}
For a profinite group $G$ and a prime number $p$ one defines the completed $\F_p$-group algebra of $G$ by
\begin{equation}
\label{eq:compFp}
    \F_p\dbl G\dbr =\varprojlim_{U\triangleleft_\circ G} \F_p G/U,
\end{equation}
where the inverse limit is running over all open normal subgroups of $G$. By construction, it is a compact $\F_p$-algebra. Indeed, the
assignment
\begin{equation}
    \label{eq:func1}
    \F_p\dbl\argu\dbr\colon \bool\longrightarrow\CVec
\end{equation}
is a functor from the category of \emph{boolean sets} (=profinite sets) to the category of compact $\F_p$-vector spaces $\CVec$, which is the
left adjoint of the forgetful functor $\mathbf{f}\colon\CVec\longrightarrow\bool$ and thus, commutes with inverse limits (cf. \cite[Chapter V Section 5]{MacLane}).
The $\F_p$-algebra $\F_p\dbl G\dbr$ comes equipped with a canonical augmentation
$\eps\colon\F_p\dbl G\dbr\to\F_p$. Put $\uomega_G=\kernel(\eps)$. Then
$F^m(\F_p\dbl G\dbr)=\uomega_G^m$, makes $\F_p\dbl G\dbr$ canonically a filtered compact $\F_p$-algebra. The associated graded $\F_p$-algebra
$\gr_\bullet\F_p\dbl G\dbr$ has homogeneous components $\gr_k\F_p\dbl G\dbr$, $k\geq 0$, consisting of compact $\F_p$-vector spaces. Similarly, the restricted Zassenhaus $\F_p$-Lie algebra $\gr_\bullet^{[p]}(G)$ is a graded restricted $\F_p$-Lie algebra with homogeneous components 
\begin{equation}
    \label{eq:prop1}
    \gr_k^{[p]}(G)=\uD_k(G,\F_p)/\uD_{k+1}(G,\F_p),
\end{equation}
consisting of compact $\F_p$-vector spaces, where
\begin{equation}
    \label{eq:prop2}
    \uD_k(G,\F_p)=G\cap (1+\uomega_G^k)=\cl(\gamma_k^{[p]}(G))
\end{equation}
(cf. Definition~\ref{def:gamma}).

\begin{fact} 
Let $G$ be a profinite group. Then $\FG$ is a separated compact $\F_p$-algebra if, and only if, $G$ is a pro-$p$ group.
\end{fact}

\begin{proof}
If $\overline{G}$ is a finite $p$-group, then there exists $m\geq 0$ such that $\omega_{\overline{G}}^m=\{0\}$, showing that for a pro-$p$ group $G$, $\FG$ is  separated. Let $G$ be a profinite group and assume that $\FG$ is  separated. In particular, 
$\bigcap_{k\geq 1} \uD_k(G,\F_p)=\{1\}$.
Suppose that $G$ is not a pro-$p$ group. Then there exists a prime $q$ different from $p$, such that its pro-$q$ Sylow subgroup $Q\in\Syl_q(G)$ is non-trivial, i.e., 
$Q\not=\{1\}$. As $\uD_k(Q,\F_p)=\uD_{k+1}(G,\F_p)$ for all $k\geq 1$, one concludes that
$Q=\bigcap_{k\geq 1} \uD_k(Q,\F_p)\subseteq 
\bigcap_{k\geq 1} \uD_k(G,\F_p)$, a contradiction, showing that $G$ has to be pro-$p$.
\end{proof}

As for discrete groups we  sat that a closed subgroup $H$ of a pro-$p$ group $G$ is 
\emph{strictly $p$-embedded} if for all $k\geq 0$ one has $\uD_k(H,\F_p)=\uD_k(G,\F_p)\cap H$.
Obviously, the closed subgroup $H$ of the pro-$p$ group $G$ is strictly $p$-embedded if, and only if,
the mapping of restricted $\F_p$-Lie algebras
$\gr_\bullet^{[p]}(\iota)\colon\gr_\bullet^{[p]}(H)\to\gr_\bullet^{[p]}(G)$ is injective, where
$\iota\colon H\to G$ denotes the canonical inclusion.

\begin{proposition}
\label{prop:se-prop}
    Let $H$ be a closed subgroup of the finitely generated pro-$p$ group $G$. Then the following are equivalent.
    \begin{itemize}
        \item[\textup{(i)}] $H$ is strictly embedded in $G$.
        \item[\textup{(ii)}] The inclusion map $\F[H]\to \F[G]$ is filtered free (for the augmentation filtration).
    \end{itemize}
\end{proposition}
\begin{proof} The proof of Propositon~\ref{prop:strictly-embedded}
can be tranferred verbatim.
\end{proof}
Theorem~\ref{thm:amalgamated} and Corollary~\ref{cor:HNN} have  obvious pro-$p$ analogues. In order to state them properly
we will denote by $\ast^{(p)}$ the amalgamated free product in the category of pro-$p$ groups.
Let $G_1$ and $G_2$ be pro-$p$ groups, and let $\lambda_i\colon H\to G_i$ be continuous homomorphisms of pro-$p$ groups. Then $G=G_1\ast_H^{(p)} G_2$ together with the continuous homomorphism $j_i\colon G_i\to G$, $i=1,2$, is the unique pro-$p$ group with the following property: For every pair of continuous homomorphisms $\phi_i\colon G_i\to X$, $i=1,2$,  of pro-$p$ groups satisfying $\phi_1\circ\lambda_1=\phi_2\circ\lambda_2$ there exists a unique continuous homomorphism of pro-$p$ groups $\phi\colon G\to X$ satisfying $\phi\circ j_{i}=\phi_{i}$. Note that, by definition, $G$ coincides with the pro-$p$ completion of the free product with amalgamation $G_\circ=G_1\ast_H G_2$.

\begin{theorem}
    \label{thm:prop}
    Let $G_1$ and $G_2$ be finitely generated pro-$p$ groups, and let $\lambda_i\colon H\to G_i$, $i=1,2$ be strict $p$-embeddings. Then for 
    $G=G_1\ast_H^{(p)} G_2$, $j_i\colon G_i\to G$, $i=1,2$, are
    strict $p$-embeddings.
    Moreover, for $G_\circ=G_1\ast _H G_2$, the pro-$p$ group $G$ coincides with
    the pro-$p$ completion of $G_\circ$ and
    \begin{equation}
        \label{eq:propeq}
\gr_\bullet(G)=\gr_\bullet(G_1)\ast_{\gr_\bullet(H)}
        \gr_\bullet(G_2).
    \end{equation}
\end{theorem}

\begin{proof}
By Theorem~B, $G_\circ$ is residually $[p]$, and by hypothesis
$G_\circ/\gamma_1^{[p]}(G_\circ)$ is finite implying that
$G_\circ$ is residually $p$-finite (cf. Proposition~\ref{prop:resp}).
Therefore, one has a canonical isomorphism 
$G\cong\varprojlim_{k\geq 1} G_\circ/\gamma_k^{[p]}(G_\circ)$.

For $k\geq 1$ put $G_k=G_\circ/\gamma_k^{[p]}(G_\circ)$ and let $\tau_K\colon G_\circ\to \F_pG_\circ/\omega_{G_\circ}^k$ denote the canonical map. Then one has $\gamma_k^{[p]}(G_{\circ})=\tau_k^{-1}(\{1\})$
and $\tau_k$ induces a canonical multiplicative map
\begin{align}
    \tau_k^\prime\colon & G_k\longrightarrow
    \F_pG_\circ/\omega_{G_\circ}^k\label{eq:taukprime}\\
    \intertext{and thus a surjective homomorphism of
    $\F_p$-algebras}
    \tau_k^{\prime\prime}\colon& \F_pG_k
\longrightarrow\F_pG_\circ/\omega_{G_\circ}^k.
\label{eq:taukcirc}
\end{align}
Thus applying $\varprojlim_{k\geq 1}$ yields a map
$\tau^{\prime\prime}\colon\FG\longrightarrow
\varprojlim_{k\geq 1} \F_pG_\circ/\omega_{G_\circ}^k$.
Since the groups $G_k=G_\circ/\gamma_k^{[p]}(G_\circ)$ are finite 
$p$-groups, there exists an integer $n_k$ such that 
$\omega_{G_k}^{n_k}=\{0\}$. 
This implies the existence of a canonical  map
    $\rho_k\colon\F_pG_\circ/\omega_{G_\circ}^{n_k}
    \longrightarrow \F_p G_k$.
Applying $\varprojlim_{k\geq 1}$ yields a homomorphism of $\F_p$-algebras
$\rho\colon \varprojlim_{k\geq 1} \F_p G_\circ/\omega_{G_\circ}^k\to\FG$, which is the inverse of $\tau^{\prime\prime}$.
Hence $\tau^{\prime\prime}$ is an isomorphism, and reasoning as in the proof of
Theorem~\ref{thm:amalgamated} yields an isomorphism
$\gr_\bullet^{[p]}(G)=\gr_\bullet^[p](G_\circ)$.
\end{proof}

Let $G_1$ be a pro-$p$ group, $H\subseteq G_1$
a closed subgroup of $G_1$ and $\varphi\colon H\to H$
a continuous isomorphism. Then we denote by
\begin{equation}
    \label{eq:proHNN}
    G=G_1\ast^{(p)}_{H,\varphi,t}
\end{equation}
the pro-$p$ HNN-extension of $(G,H,\varphi)$ with stable letter $t$, i.e.,
$G$ is the unique pro-$p$ group containing a distinguished element $t$ with the property
that for any homomorphism of pro-$p$ groups
$\phi_\circ\colon G_1\to X$, where $X$ is a pro-$p$ group containing a closed subgroup $Y$ isomorphic to $H$
and an element $x\in X$ such that conjugation by $x$
induces the isomorphism $\varphi$ on $Y$, there exists a unique homomorphism $\phi\colon G\to X$ satisfying
$\phi\vert_{G_1}=\phi_\circ$ and $\phi(t)=x$.
Clearly, as before $G_1\ast^{(p)}_{H,\varphi}$ coincides with the pro-$p$ completion of $G_1\ast_{H,\varphi,t}$
For pro-$p$ groups one obtains the following:

\begin{corollary}\label{cor:propHNN} 
Let $G_1$ be a finitely generated pro-$p$ group,
let $H\leq G_1$ be a strictly $p$-embedded subgroup and $\varphi:H\to H$ an automorphism that acts trivially on $H/\gamma^{[p]}_2(H)$. Let $$G=G_1\ast_{H,\varphi}^{(p)}.$$
Then
$$\gr_\bullet^{[p]}(G)=\gr_\bullet^{[p]}(G_1)\ast_{\gr_\bullet^{[p]}(H),d}$$
where $d$ is the derivation of $\gr_\bullet^{[p]}(H)$ induced by $\varphi$.
\end{corollary}

\begin{proof} 
Note that as $H$ is strictly $p$-embedded in the finitely generated pro-$p$ group $G$, it is also finitely generated.
Let $G_2=\langle t\rangle\ltimes_{\varphi} H$, where $t$ is the stable letter and the $t$ action is given by $\varphi$. Then $G_2$ is an almost $p$-direct product. So $H$ is strictly $p$-embedded in $G_2$ by Lemma~\ref{lem:FRp}. Moreover
\begin{equation}
    \label{eq:deco1}
G_1\ast_{H,\varphi,t}=G_1\ast_HG_2.
\end{equation}
Therefore, as $G_2$ is residually $[p]$ and as
$G_2/\gamma_2^{[p]}(G_2)$ is finite, $G_2$ is residually $p$-finite (cf. Proposition~\ref{prop:resp}). Let $(G_2)^\vee_p$ denote its 
pro-$p$ completion. Then \eqref{eq:deco1} implies that
\begin{equation}
    \label{eq:deco2}
    G_1\ast_{H,\varphi}^{(p)}\cong G_1\ast_{H}^{(p)}
    (G_2)^\vee_p
\end{equation}
and as - by hypothesis - $H$ is strictly $p$-embedded in $G_1$ and $(G_2)^\vee_p$, Theorem \ref{thm:prop} implies that 
$$\gr_\bullet^{[p]}(G)=\gr_\bullet^{[p]}(G_1)\ast_{\gr_\bullet^{[p]}(H)}\gr_\bullet^{[p]}(G_2).$$
Hence the obvious modification of the presentation of $\gr_\bullet^{[p]}(G)$ gives the desired result.
\end{proof}

\subsection{Criteria for strict $p$-embeddings}
\label{ss:crit}
Some applications require criteria ensuring
for a (closed) subgroup $H$ of a (pro-$p$) group $G$ to be strictly $p$-embedded. 

\begin{lemma}
\label{lem:iteratealmost} Let $p$ be a prime and 
let $G$ be a (pro-$p$) group with (closed) subgroups $N,H,T$ such that $N\leq T$, $T\cap H=1$, $N$ and $H_1=NH$ are normal in $G$ and  the products $H_1=N\rtimes H$ and $G/N=T/N\ltimes H_1/N$ are almost $p$-direct products. Then $H$ is strictly $p$-embedded in $G$.
\end{lemma}

\begin{proof} We have to check that for any $i\geq 1$, $\cl(\gamma_i^{[p]}(H))=\cl(\gamma_i^{[p]}(G))\cap H$.
The fact that $G/N=T/N\ltimes H_1/N$ is an almost $p$-direct product implies that $H_1/N$ is strictly $p$-embedded in $G/N$ (cf. Lemma~\ref{lem:FRp}), so
\begin{equation}\label{eq:tectwisted1}
\begin{aligned}
    (\cl(\gamma_i^{[p]}(G))N\cap H_1)/N&=\cl(\gamma_i^{[p]}(G/N))\cap(H_1/N)=\cl(\gamma_i^{[p]}(H_1/N))\\
    &=\cl(\gamma_i^{[p]}(H_1))N/N.
    \end{aligned}
\end{equation}
As $H_1=N\rtimes H$ is also an almost $p$-direct product, one concludes that $\cl(\gamma_i^{[p]}(H_1))=\cl(\gamma_i^{[p]}(N))\cl(\gamma_i^{[p]}(H))$. Hence
\begin{equation}\label{eq:tectwisted2}
   \cl(\gamma_i^{[p]}(H_1))N\cap H=
   \cl(\gamma_i^{[p]}(H)). 
\end{equation}
Therefore using (\ref{eq:tectwisted1}) and (\ref{eq:tectwisted2}) one obtains
\[
\begin{aligned}
\cl(\gamma_i^{[p]}(G))\cap H&\leq\cl(\gamma_i^{[p]}(G))N\cap H=(\cl(\gamma_i^{[p]}(G))N\cap H_1)\cap H\\
&=\cl(\gamma_i^{[p]}(H_1))N\cap H=
\cl(\gamma_i^{[p]}(H))
\end{aligned}\]
and as $\cl(\gamma_i^{[p]}(H))\leq\cl(\gamma_i^{[p]}(G))\cap H$ is obvious, they must indeed coincide.
\end{proof}

For our next criteria we will make use of the following.

\begin{lemma}
\label{lem:help}
Let $\F$ be a field and let 
$\alpha_\bullet\colon A_\bullet\to B_\bullet$ be a homomorphism of graded connected $\F$-algebras
(resp. $\F$-Lie algebras, restricted $\F$-Lie algebras
in case that $\chrs(\F)=p>0$, etc.). Assume further
that
\begin{itemize}
    \item[\textup{(i)}] $A_\bullet$ is $1$-generated, i.e.,
    $A_\bullet=\langle A_1\rangle$ and $\dim_{\F}(A_1)<\infty$;
    \item[\textup{(ii)}] $\alpha_1\colon A_1\to B_1$ is injective;
    \item[\textup{(iii)}] $\langle\image(\alpha_1)\rangle\cong A_\bullet$.
\end{itemize}
    Then $\alpha_\bullet\colon A_\bullet\to B_\bullet$ is injective.
\end{lemma}

\begin{proof}
Note that by (i), $A_\bullet$ is of finite type, i.e,
$\dim_{\F}(A_k)<\infty$ for all $k\geq 0$. Put
$C_\bullet=\langle\image(\alpha_1\rangle$. Then, by definition, $C_\bullet$ is $1$-generated and thus of finite type, i.e., $\dim_\F(C_k)<\infty$ for all $k>0$. The induced map 
$\alpha_\bullet^\vee\colon A_\bullet\to C_\bullet$ is
a homomorphism of $1$-generated graded connected $\F$-algebras
(resp. $\F$-Lie algebras, restricted $\F$-Lie algebras, etc. and $\alpha_1^\vee$ is an isomorphism
by (ii).
Thus $\alpha_\bullet^\vee$ is surjective, i.e.,
$\alpha_k^\vee$ is surjective for all $k>0$. Hence
(iii) implies that $\kernel(\alpha^\vee_k)=0$ for all
$k>0$, and $\alpha_\bullet^\vee$ is an isomorphism.
\end{proof}

\begin{proposition}
    \label{prop:help} Let $p$ be either $0$ or a prime.
    Let $H$ be a finitely generated (closed) subgroup of the (pro-$p$) group $G$ and let $\iota\colon H\to G$ denote the canonical inclusion. Assume further that
    \begin{itemize}
        \item[\textup{(i)}] $\gr_1^{[p]}(\iota)\colon \gr_1^{[p]}(H)\to\gr_1^{[p]}(G)$ is injective,
        \item[\textup{(ii)}] $\langle\image(\gr_1^{[p]}(\iota)\rangle
        \cong\gr_\bullet^{[p]}(H)$.
    \end{itemize}
    Then $H$ is strictly $p$-embedded in $G$.
\end{proposition}

\begin{proof}
    It is well-known that $\gr_\bullet^{[p]}(H)$
    is a $1$-generated restricted $\F_p$-Lie algebra
    for $p>0$, and a $1$-generated $\Q$-Lie algebra
    in case that $p=0$. Hence the finite generation
    of $H$ implies (i) of Lemma~\ref{lem:help}.
    Moreover, (ii) of Lemma~\ref{lem:help} follows from (i), and
    (iii) of Lemma~\ref{lem:help} follows from (ii).
  Hence Lemma~\ref{lem:help} implies that
  $\gr_\bullet^{[p]}(\iota)$ is injective.
\end{proof}

From Proposition~\ref{prop:help} one concludes the following 

\begin{proposition}
    \label{prop:crit}
    Let $p$ be either $0$ or a prime.
    Let $H$ be a finitely generated (closed) subgroup of the (pro-$p$) group $G$ and let $\iota\colon H\to G$ denote the canonical inclusion. Assume further that 
    \begin{itemize}
        \item[\textup{(i)}] 
    there exists a continuous homomorphism $\tau\colon G\to H/\gamma_2^{[p]}(H)$ such that
    $\iid_{\gr_1^{[p]}(H)}=\gr_1^{[p]}(\tau)\circ\gr_1^{[p]}(\iota)$;
    \item[\textup{(ii)}]  $\langle\image(\gr_1^{[p]}(\iota)\rangle
        \cong\gr_\bullet^{[p]}(H)$.
    \end{itemize}
    Then $H$
    is strictly $p$-embedded in $G$.
\end{proposition}

\section{An alternative proof of residual torsion free nilpotence of free products amalgamated along retracts}\label{sec:rtfn}

Considering the case when $p=0$, Theorem \ref{teo:strictly} implies that the free amalgamated product of residually torsion free nilpotent groups along a strictly embedded common subgroup is also residually torsion free nilpotent. In particular, this is the case when the amalgamated subgroup is a retract. In this section we provide a purely group theoretical direct proof of this fact. In the literature there are similar results for amalgamations along retracts, mainly for residual finiteness or in the  prime case \cite{BolerEvans}, \cite{Sokolov}. See also Theorem A in \cite{CapraceMarquis} where residual torsion free nilpotence is considered assuming that the groups $G_1$ and $G_2$ are torsion free nilpotent and the amalgamated subgroup $H$ is abelian.

A classical result by Philip Hall says that if $G$ is a group with a normal subgroup $K$ such that $K$ and $G/[K,K]$ are nilpotent, then so is $G$. This has been extended by Baumslag who shows

\begin{lemma} \cite[Proposition 1]{BM} \label{lemma_split}
    Let $G$ be a group with a residually torsion free nilpotent normal subgroup $K$. Assume that the quotient $G/K$ is torsion free nilpotent and that $G/[K,K]$ is nilpotent. Then $G$ is residually torsion free nilpotent.
\end{lemma}

Using this we obtain

\begin{lemma}\label{lemma_nilpotent}
     Let $G_1$, ${G_2}$ be torsion free nilpotent groups with a common retract $H$. Then their free amalgamated product $G=G_1\ast_H {G_2}$ is residually torsion free nilpotent.
\end{lemma}
\begin{proof} As $H$ is a retract in both $G_1$ and ${G_2}$   the kernels of the projections $G_1\to H$ and ${G_2}\to H$
 are normal subgroups such that ${G_1} = H \ltimes K_{1}$ and ${G_2}=H \ltimes K_{2}$. As ${G_1}$ and ${G_2}$ are torsion free nilpotent, so are $K_{1}$ and $K_{2}$ so
the group $K=K_{1} \ast K_{2}$ is residually torsion free nilpotent by  \cite{Malcev}. Moreover, it is a standard fact that  $G$ also decomposes as a semidirect product $G=H \ltimes K$ (this can be easily seen by checking their presentations). Moreover $H=G/K$ is torsion free nilpotent. We claim that the quotient $G/[K,K]$ is nilpotent. To see it, note that 
$$K/[K,K]=K_{1}/[K_{1},K_1]\times K_{2}/[K_{2},K_2]$$
and, as ${G_1}$ and ${G_2}$ are nilpotent, so are ${G_1}/[K_{i},K_i]=H\ltimes K_{i}/[K_{i},K_i]$ for $i=1,2$. The direct product of these two groups is also nilpotent and 
$G/[K,K]=H\ltimes K/[K,K]$ embeds into $${G_1}/[K_1,K_1]\times{G_2}/[K_2,K_2]=(H\times H)\ltimes K/[K,K]$$ via the diagonal map in the first component so it is also nilpotent. Hence Lemma \ref{lemma_split} yields the claim. \end{proof}

\begin{theorem}\label{thm:restfreenilp}
     Let ${G_1}$, ${G_2}$ be residually torsion free nilpotent groups with a common retract $H$. Then their free product with amalgamation $G={G_1}\ast_H {G_2}$ is residually torsion free nilpotent.
\end{theorem}
\begin{proof}   Let $1\neq g \in G$. We need to find a normal subgroup $N \trianglelefteq G$ that does not contain $g$ so that $G/N$ is torsion free nilpotent. As in Lemma \ref{lemma_nilpotent}, we denote by $K_1$ and $K_2$ the kernels of the projections ${G_1}\to H$ and ${G_2}\to {G_2}$ so ${G_1} = H \ltimes K_1$ and ${G_2}=H \ltimes K_2$ and 
$G = H \ltimes K$
with $K=K_1\ast K_2$. If $g\not\in K$, then $1\neq gK\in G/K=H$ and as $H$ is residually torsion free nilpotent we can find $N\trianglelefteq G$ with $G/N$ torsion free nilpotent and $g\not\in N$. So we may assume $g\in K=K_1\ast K_2$. Therefore we can write $g$ as a product 
$g=g_0g_1 \dots g_t$ where each $g_i$ is alternatively non trivial element of $K_1$ or $K_2$. 

Now, we consider the torsion free central series $\gamma_n^{[0]}$ of the groups ${G_1}$ and ${G_2}$. As both are residually torsion free nilpotent, one has
$$1=\cap\gamma_n^{[0]}({G_1})=\cap\gamma_n^{[0]}({G_2}).$$
Therefore, we may choose an $n$ big enough so that for all the elements $g_i$ in the above decomposition of $g$, one has:
\begin{itemize}
\item if $g_i\in K_1$, then $g_i\not\in\gamma_n^{[0]}({G_1})$,
\item if $g_i\in K_2$, then $g_i\not\in\gamma_n^{[0]}({G_2})$.
\end{itemize}
Moreover, $\overline{{G_1}}={G_1}/\gamma^{[0]}_n({G_1})$ and $\overline{{G_2}}={G_2}/\gamma^{[0]}_n({G_2})$ are both torsion free nilpotent. 
The hypothesis that $H$ is a retract in both $G_1$ and $G_2$ implies
$$\gamma^{[0]}_n(H)=H\cap\gamma^{[0]}_n({G_1})=H\cap\gamma^{[0]}_n({G_2})$$
so $\overline{H}=H/\gamma^{[0]}_n(H)$
 is a retract in both $\overline{{G_1}}$ and $\overline{{G_2}}$. Consider the group $\overline{G}=\overline{{G_1}}\ast_{\overline{H}}\overline{{G_2}}$. By Lemma \ref{lemma_nilpotent} it is residually torsion free nilpotent. Moreover, the projections ${G_1}\to \overline{{G_1}}$ and ${G_2}\to\overline{{G_2}}$ restrict to the projection $H\to\overline{H}$ and there is an induced epimorphism
 $$\tau:G={G_1}\ast_H{G_2}\to\overline{G}=\overline{{G_1}}\ast_{\overline{H}}\overline{{G_2}}.$$
Then $\tau(g)=\tau(g_1) \dots\tau(g_t)$ and the choice of $n$ implies that for each $i$, $\tau(g_i)$ lies either in $\tau(K_1)$ or in $\tau(K_2)$ but it is non-trivial. Therefore, $\tau(g_i)\not\in\overline{H}$ so for the normal form theorem for free amalgamated products, we get $\tau(g)\neq 1$. As $\overline{G}$ is residually torsion free nilpotent, this means that we can find $\overline{N}\trianglelefteq\overline{G}$ such that $\tau(g)\not\in\overline{N}$ and $\overline{G}/\overline{N}$ is torsion free nilpotent. Therefore, it is enough to consider the preimage $N$ of $\overline{N}$ in $G$.
\end{proof}

\section{Almost retracts and cohomological $p$-completeness.}\label{sec:chpc}

Let $p$ be a prime. For a group $G$, let $G_p^\vee$ denote the pro-$p$ completion of $G$, i.e.
\begin{equation}\label{eq:propcom}
{G}_p^\vee=\varprojlim_{U\trianglelefteq_p G} G/U,
\end{equation}
where $U\trianglelefteq_p G$ means that $U$ is an open normal subgroup of $p$-power index in $G$.  Assume that $G$ is finitely generated and residually $p$-finite. Then there is an injection
\begin{equation}
\label{eq:propcom2}
G\hookrightarrow G_p^\vee
\end{equation}
that induces a map
$$H^*_{\mathrm{cont}}(G_p^\vee,\F_p)\to H^*(G,\F_p).$$

\begin{definition} We say that a group $G$ is \emph{cohomologically $p$-complete}
if it is finitely generated, residually $p$-finite, and the map $H^m_{\mathrm{cont}}(G_p^\vee,\F_p)\to H^m(G,\F_p)$ is an isomorphism for all positive integers $m$.
\end{definition}

In the setting of \cite[\S1.2.6]{Serre} cohomologically $p$-complete groups coincide with pro-$p$ \emph{good} groups, i.e.,
discrete groups $G$ for which the pro-$p$ completion 
\eqref{eq:propcom2} is good.
The group $G_p^\vee$ is in fact the completion of $G$ with respect to the pro-$p$ topology, which is the topology having the set of normal subgroups $T\trianglelefteq G$ with $G/T$ a finite $p$-group as fundamental  system of neighborhoods of the identity. 

In \cite{AF}, Aschenbrebber and Friedl give conditions on a graph of finitely generated groups $\mathcal{G}$ that imply that the fundamental group $\pi(\mathcal{G})$ is cohomologically $p$-complete. The next result is a direct consequence of their results but for completeness we include a proof of the part of the argument which is not stated so explicitly in \cite{AF}. 

We will need the following notation. Let $\mathcal{G}$ be a graph of groups with underlying graph $Y$ and assume that the vertex and edge groups are strictly $p$-embedded in its fundamental group $G=\pi_1(\mathcal{G})$. We will denote vertex groups by $G_v,G_e$ where $v$ is a vertex and $e$ an edge of $Y$ and if $v$ is the endpoint of $e$, we denote by $f_e$ the map $f_e\colon G_e\to G_v$.
For any $n\geq 1$, let $\mathcal{G}_n$ be the graph of groups with the same underlying graph as $\mathcal{G}$ but with vertex groups $G_v/\gamma_n^{[p]}(G_v)$ where $G_v$ is a vertex group of $\mathcal{G}$ and edge groups $G_e/\gamma_n^{[p]}(G_e)$ where $G_e$ is a vertex group of $\mathcal{G}$. As vertex and edge groups are strictly $p$-embedded in $G$ we have
$$\gamma_n^{[p]}(G)\cap G_v=\gamma_n^{[p]}(G_v)$$
and
$$\gamma_n^{[p]}(G)\cap G_e=\gamma_n^{[p]}(G_e)$$
and from this one gets the compatibility conditions needed for the graph of groups $\mathcal{G}_n$ to be well defined. Moreover, there is a canonical epimorphism
$$\tau_n:\pi_1(\mathcal{G})\to\pi(\mathcal{G}_n).$$

\begin{theorem}\cite[Corollary 3.8, Corollary 5.12]{AF}\label{thm:efficient} Let $\mathcal{G}$ be a graph of finitely generated groups and assume that:
\begin{itemize}
\item[\textup{(i)}] vertex and edge groups are strictly $p$-embedded in $G=\pi_1(\mathcal{G})$,

\item[\textup{(ii)}] edge groups are closed in the pro-$p$ topology of the vertex groups.

\item[\textup{(iii)}] for any $n\geq 1$, $G=\pi_1(\mathcal{G}_n)$ is residually $p$-finite.
\end{itemize}
Then the vertex groups are closed in the pro-$p$ topology of $G$. If moreover
\begin{itemize}
\item[\textup{(iv)}] vertex and edge groups are cohomologically $p$-complete, 
\end{itemize}
then $G$ is cohomologically $p$-complete.
\end{theorem}
\begin{proof} We prove the first claim. The fact that (i)-(iv)   imply the cohomologically $p$-completeness
of $G$
is \cite[Corollary 5.12]{AF}.

Let $G_v$ be a vertex group and take $g\in G-G_v$. We are going to show that there is some normal subgroup $K\trianglelefteq G$ with $G/K$ a finite $p$-group such that $g\not\in G_vK$, this will imply the claim. 

Without lost of generality we may take the vertex $v$ as the base point $v_0$ in the construction of the fundamental group $\pi(\mathcal{G})$. Then, the element $g$ can be represented by a reduced path in the underlying graph $Y$ of $\mathcal{G}$:
$$g=g_0e_1g_1e_2\ldots e_kg_k$$
where  the $e_i$ form a closed path in $Y$ starting and ending in $v=v_0$, $g_0\in G_{v}$, and for each $1\leq i\leq k$, if  $v_i$ is the endpoint of $e_i$,  $g_{i}\in G_{v_{i}}$ (we are using the same notation as in \cite[Section 1.2.2]{AF}). Moreover, the fact that the path is reduced means that either $n=0$ and $g_0\neq 1$ or $n>0$ and for any $i$ such that $e_{i+1}=\overline{e}_i$, we have $g_i\not\in f_{e_i}(G_{e_i})$. Observe that the condition $g\not\in G_v$ is equivalent to $k>1$ in this reduced path.

We claim that conditions (i) and (ii) imply that there is some $n$ such that $\tau_n(g)\not\in G_v/\gamma_n^{[p]}(G_v)$. We need to show that we may choose an $n$ so that the path
 $$\overline{g}_0e_1\overline{g}_1e_2\ldots e_k\overline{g}_k$$
 is still reduced where $\overline{g}_i$ is the image of $g_i$ under the projection map $G_{v_i}\to G_{v_i}/\gamma_n^{[p]}(G_{v_i})$. For each $1\leq i\leq k$ and each $n$, (i) implies that we have a commutative diagram
$$
\begin{tikzcd}
G_{e_i}\arrow[r,"f_{e_i}"]\arrow[d,two heads]&G_{v_i}\arrow[d,two heads]\\
G_{e_i}/\gamma_n^{[p]}(G_{e_i})\arrow[r,"\overline{f}_{e_i}"]&G_{v_i}/\gamma_n^{[p]}(G_{v_i})\\
\end{tikzcd}
$$
and the image of $G_{e_i}/\gamma_n^{[p]}(G_{e_i})$ under $\overline{f}_{e_i}$ is $f_{e_i}(G_{e_i})\gamma_n^{[p]}(G_{v_i})/\gamma_n^{[p]}(G_{v_i})$. Condition (ii) implies that we may choose $n$ big enough so that 
$$g_i\not\in f_{e_i}(G_{e_i})\gamma_n^{[p]}(G_{v_i})$$
for all the indices $i$, where it is required.  

Now we are going to use the following observation (\cite[Lemma 3.10]{AF}). If $L$ is residually $p$-finite and $S\leq L$ is a finite subgroup, then for any $x\in L-S$ there is a normal subgroup $\hat{K}\triangleleft L$ with $L/\hat{K}$ a finite $p$-group such that $x\not\in S\hat{K}$. 
As $S$ is finite and $L$ residually $p$-finite, there is some $r$ such that $S\cap\gamma^{[p]}_r(L)=1$. If $x\not\in S\gamma^{[p]}_r(L)$, then we only have to take $\hat{K}=\gamma^{[p]}_r(L)$. So we assume $x\in S\gamma^{[p]}_r(L)$. Then we may write $x=sy$ with $s\in S$, $y\in\gamma^{[p]}_r(L)$ and as $L$ is residually $p$-finite, there is some $m>r$ with $y\not\in \gamma^{[p]}_m(L)$. So $x=sy\not\in S\gamma^{[p]}_m(L)$ and we may take $\hat{K}=\gamma^{[p]}_m(L)$.

We apply this for $L=\pi_1(\mathcal{G}_n)$ which is residually $p$-finite by (iii), $S=G_v/\gamma_n^{[p]}(G_v)$ and $x=\tau_n(g)$ and get a subgroup $\hat{K}$. Then we only have to let $K$ be the inverse image of $\hat{K}$ under $\tau_n$. As $\tau_n(g)\not\in S\hat{K}$, $g\not\in G_vK$ and $G/K\cong L/\hat{K}$ is a finite $p$-group.
\end{proof}

\begin{definition}\label{def:almostretract} Let $G$ be a residually $p$-finite group and $H\leq G$ a subgroup. We say that $H$ is an \emph{almost $p$-retract} of $G$ if
there is some $K\leq G$ such that $K\cap H=1$ and for any $n\geq 1$
\begin{equation}\label{eq:almostpretract}\gamma^{[p]}_n(G)=(\gamma^{[p]}_n(G)\cap K)\gamma^{[p]}_n(H).\end{equation}
\end{definition}

Observe that in the definition of an almost $p$-retract  we require in particular $G=KH$ (taking $n=1$). Moreover, almost $p$-retracts are strictly $p$-embedded because (\ref{eq:almostpretract}) implies 
$$\gamma^{[p]}_n(G)\cap H=\gamma^{[p]}_n(H).$$
Another useful observation is that retracts are almost $p$-retracts for any $p$ because of (\ref{eq:retract}). In fact,  in the case of almost $p$-retracts
 we are not asking  the subgroup $K$ to be normal. For example, if $G=H\ltimes K$ is an almost $p$-direct product, then  both $H$ and $K$ are almost $p$-retracts. 
 
\begin{lemma}\label{lem:proptop}  Let $H$ be an almost $p$-retract of the group $G$. Assume further that $G$ is residually $p$-finite and that $G$ and $H$ are both finitely generated. Then 
\begin{itemize}
\item[\textup{(i)}] $H$ is closed in the pro-$p$ topology of $G$ and
\item[\textup{(ii)}] the pro-$p$ topology of $G$ induces the pro-$p$ topology of $H$.
\end{itemize}
\end{lemma}
\begin{proof} Item (ii) is obvious because as remarked above the condition $\gamma_n^{[p]}(G)\cap H=\gamma_n^{[p]}(H)$ follows from the definition of almost $p$-retract. For item (ii) we check that $H=\cap_n(\gamma_n^{[p]}(G)H)$. Let $g\in\cap_n(\gamma_n^{[p]}(G)H)$, then as $G=KH$ with $K\cap H=1$ we may write $g=kh$ where $k\in K$, $h\in H$ and $h,k$ are uniquely determined. We also have
$$\gamma^{[p]}_n(G)=(\gamma^{[p]}_n(G)\cap K)\gamma^{[p]}_n(H)$$
for every $n\geq 1$ so
$$\gamma^{[p]}_n(G)H=(\gamma^{[p]}_n(G)\cap K)H$$
so $g\in\gamma^{[p]}_n(G)H$ implies $k\in\gamma^{[p]}_n(G)\cap K$ for any $n$ and as $G$ is residually $p$-finite this implies $k=1$. Therefore $g=h\in H$ and we get
$$H=\cap_n(\gamma_n^{[p]}(G)H)$$
as we wanted to prove.
\end{proof}

Finally, we are able to prove Theorem \ref{thm:pcomplete}.

\medskip

\noindent\emph{Proof of Theorem \ref{thm:pcomplete}.} By definition of almost $p$-retract, the edge groups are strictly $p$-embedded in the vertex groups, moreover this condition is also true for each $n\geq 1$ for the graph of groups $\mathcal{G}_n$ where $G=\pi_1(\mathcal{G})$. Therefore Theorem \ref{teo:strictly} implies that each $\pi_1(\mathcal{G}_n)$ is residually $p$-finite. We also deduce from Theorem \ref{teo:strictly}  that the vertex groups are strictly $p$-embedded in $G$. So conditions (i) and (iii) in Theorem \ref{thm:efficient} hold. Lemma \ref{lem:proptop} implies that also (ii) holds and (iv) is part of the hypothesis. \qed

\section{An example: The 2-Zassenhauss algebra of twisted RAAGs}
\label{s:TRAAG}

A \emph{mixed graph} $\Gamma=(\euV,\euE,\euA)$ is a simplicial graph where some of the edges are oriented and some are not. We say that an unoriented edge is an \emph{edge} and denote the set of all edges by $\euE$, and  that an oriented edge  is an \emph{arrow} and denote the set of all arrows by $\euA$.
We may consider the set of edges as a subset
of $\euP_2(\euV)$, the set of all subsets of $\euV$
of cardinality $2$. While the set of arrows may be considered as a subset of $\euV\times\euV$. 
In order to avoid complications, we also assume that for any arrow $\eua=(u,v)$, one has 
$\{u,v\}\not\in\euE$ and $u\not=v$.

A mixed graph $\Gamma$ without arrows will be called a na\"ive graph. Every mixed graph 
$\Gamma=(\euV,\euE,\euA)$ defines a na\"ive graph
$\ddot{\Gamma}=(\euV,\ddot{\euE})$,
where 
$\ddot{\euE}=\euE\sqcup\{\,\{u,v\}\mid (u,v)\in\euA\,\}$.
Associated to a mixed graph $\Gamma=(\euV,\euE,\euA)$ one may define the 
\emph{twisted right-angled Artin group} $G_\Gamma$ by
\begin{equation}
    \label{eq:defG}
G_{\Gamma}=\Bigg\langle\  \euV\  \Bigg\vert\ 
\begin{aligned}
[v,u]&=1\qquad \text{if 
$\{u,v\}\in\euE$,}\\
u^v&=u^{-1}\quad\text{if $(u,v)\in\euA$.} \\
\end{aligned}
\Bigg\rangle
\end{equation}
The twisted right-angled Artin group $G_\Gamma$ is torsion free if, and only if, there is no oriented cycle embedded in a clique \cite[Theorem 5.16]{Islam} and it is left orderable if, and only if, $\Gamma$ does not contain oriented cycles  \cite[Theorem 1.5]{ABP}.  

Following \cite{Blumer} we define the following 
restricted $\F_2$-Lie algebra associated to a mixed graph $\Gamma$ by 
$$L_{\Gamma}=\Bigg\langle\  \euV\ \Bigg\vert\ 
\begin{aligned}
[v,u]&=0\qquad \text{if $\{u,v\}\in\euE$,}\\
[v,u]&=u^{[2]}\quad\text{if $(u,v)\in\euA$.} \\
\end{aligned}\Bigg\rangle$$

%

A first observation  is that there is an epimorphism

\begin{equation}\label{eq:epi}\rho:L_{\Gamma}\to\gr^{[2]}_\bullet(G_{\Gamma})\end{equation}
induced by the obvious identification between generators of $G_{\Gamma}$ and of $L_{\Gamma}$.
To see it, note that if $u,v$ are linked by an edge in $\Gamma$, then $[u,v]=1$ so the corresponding elements $\overline{u},\overline{v}$ in $\gr^{[p]}_\bullet(G_{\Gamma})$ satisfy $[\overline{u},\overline{v}]=0$. And if there is an arrow $(u,v)$, then $u^v=u^{-1}$ so $[u,v]=u^{-2}$ and
$$[u,v]u^2=1.$$
As $[u,v],u^2\in\gamma_2^{[2]}(G_{\Gamma})$, this means that in the 2-Zassenhaus algebra one has
$$[\overline{v},\overline{u}]=\overline{u}^{[2]}.$$ 

In \cite[Question 4.3]{Blumer}, Blumer asks whether $L_{\Gamma}$ is the 2-Zassenhaus Lie algebra of $G_{\Gamma}$. In this section we answer this question affirmatively.

\begin{proposition}
\label{prop:twRAAGstrict}
    Let $G_{\Gamma}$ be a twisted right-angled Artin group. Then 
    for any induced subgraph $\Delta\subseteq\Gamma$ the group $G_{\Delta}$ is strictly 2-embedded in $G_{\Gamma}$.
\end{proposition}
\begin{proof} We proceed by induction on $n=|\euV(\Gamma)|$. Taking into account that being strictly 2-embedded is a transitive relation, it suffices to prove the claim for induced subgraphs $\Delta$ satisfying $\euV(\Delta)=\euV(\Gamma)\setminus\{w\}$ for some $w\in\euV(\Gamma)$. Assume first that there is some $u\in\euV(\Delta)$ such that  $(w,u),(u,w)\not\in\euA$ and $\{w,u\}\not\in\euE$, i.e., there is no arrow nor edge between $u$ and $w$. 
Let $\Delta_1\subseteq\Delta$ resp. $\Gamma_1\subseteq\Gamma$ be the subgraphs induced by $\euV(\Delta)\setminus\{u\}$ resp. $\euV(\Gamma)\setminus\{u\}$. The induction hypothesis implies that the subgroups $G_{\Delta_1}\leq G_\Delta$ and $G_{\Delta_1}\leq G_{\Gamma_1}$ are strictly 2-embedded. Moreover, one has a decomposition as a free amalgamated product
\begin{equation}\label{eq:amalgam}
    G_{\Gamma}=
G_{\Delta}\ast_{G_{\Delta_1}}
G_{\Gamma_1}.
\end{equation}
So Theorem \ref{teo:strictly} implies that 
$G_\Delta$ is strictly 2-embedded in $G_\Gamma$.

The case which remains to be considered is when for any $u\in\euV(\Delta)$ either $(w,u)$ or $(u,w)$ lie in $\euA$, or $\{u,w\}\in\euE$. 
Let $T=\langle w\rangle$, $N=\langle w^2\rangle$ and $H=G_\Delta$. Clearly, $T\cap H=1$.  For any  $u\in\euV(\Delta)$ we must have one of the following possibilities 
\begin{equation}\label{eq:cases}\end{equation}
\begin{itemize} 
    \item either $\{u,w\}\in\euE$ so $[w,u]=1$ and thus $[w^2,u]=1$,
    \item or $\eua=(u,w)\in\euA$, so $u^{w}=u^{-1}$ and thus 
    $[w^2,u]=1$,
    \item or $\eua=(w,u)\in\euA$, so $w^u=w^{-1}$ and thus $(w^2)^u=w^{-2}$.
\end{itemize}
so we deduce that $N\trianglelefteq G_\Gamma$ and also that the semidirect product $H_1=N\rtimes H$ is an almost 2-direct product. But (\ref{eq:cases}) also implies that in the quotient $G_\Gamma/N$, the element $x=wN$ normalizes $H_1/N$ so $H_1/N\triangleleft G_\Gamma/N$ and in fact the semidirect product $G_\Gamma/N=T/N\ltimes H_1/N$ is also an almost 2-direct product. This means that the hypothesis of Lemma \ref{lem:iteratealmost} are satisfied for $p=2$ thus $H=G_\Delta$ is strictly 2-embedded in $G_\Gamma$.
\end{proof}

As a by-product we also obtain:

\begin{proposition}
    Twisted right angled Artin groups are residually 2-finite.
\end{proposition}
\begin{proof} Let $\Gamma$ be a mixed graph and consider the group $G_\Gamma$, we proceed by induction on $|\euV(\Gamma)|$. Assume first that the na\"ive graph $\ddot{\Gamma}$ associated to $\Gamma$ is complete, then using (\ref{eq:cases}) above one sees that the subgroup
\[K=\langle\, w^2\mid w\in\euV(\Gamma)\,\rangle\]
is normal in $G_\Gamma$ and moreover that $G_\Gamma/K$ is 2-elementary abelian, in particular a finite 2-group. We also deduce that $K$ is free abelian, so $G_\Gamma$ is an extension of a  free abelian group by a finite 2-group so it is residually 2-finite.

If $\ddot{\Gamma}$ is not complete, choose $u,w\in\Gamma$ so that  $(w,u),(u,w)\not\in\euA$
and $\{u,w\}\not\in\euE$. Then we have a decomposition  (\ref{eq:amalgam}) as in the proof of Proposition \ref{prop:twRAAGstrict}, and the result follows by induction using 
Theorem~\ref{teo:strictly} and the fact that the group $G_{\Delta_1}$ is strictly 2-embedded in $G_{\Delta}$ and $G_{\Gamma_1}$  by 
Proposition~\ref{prop:twRAAGstrict}.
\end{proof}

\begin{theorem} 
\label{thm:Blum} Let $G_{\Gamma}$ be a twisted right angled Artin group. Then 
$$\gr^{[2]}_\bullet(G_\Gamma)=L_{\Gamma}$$
\end{theorem}
\begin{proof} Again,  we proceed by induction on $n=|\euV(\Gamma)|$. Assume first that there are 
    $u,w\in\euV(\Gamma)$ such that $(w,u),(u,w)\not\in\euA$ and $\{u,w\}\not\in\euE$. Again, we have a decomposition   (\ref{eq:amalgam}) as in the proof of Proposition \ref{prop:twRAAGstrict}, where by Proposition \ref{prop:twRAAGstrict} the group $G_{\Delta_1}$ is strictly 2-embedded in $G_{\Delta}$ and $G_{\Gamma_1}$. So
we may apply Theorem~\ref{thm:graded} and get
\[\gr_\bullet^{[2]}(G_\Gamma)=\gr_\bullet^{[2]}(G_{\Delta})\ast_{\gr_\bullet^{[2]}(G_{\Delta_1})}\gr_\bullet^{[2]}(G_{\Gamma_1}).\]
As by induction $gr_\bullet^{[2]}(G_{\Delta})=L_\Delta$, $\gr_\bullet^{[2]}(G_{\Delta_1})=L_{\Delta_1}$ and $\gr_\bullet^{[2]}(G_{\Gamma_1})=L_{\Gamma_1}$ 
this yields the desired result.

So we are left with the case when for each $w,u\in\euV(\Gamma)$, either  $(w,u)$ or $(u,w)$ 
is contained in  $\euA$ or $\{u,w\}\in\euE$. Fix some $w\in\euV(\Gamma)$ and consider the groups $N,T,H$ as at the end of the proof of Proposition \ref{prop:twRAAGstrict}. Let also $x=wN$. As observed above, we have an epimorphism
\[L_\Gamma\to\gr_\bullet^{[2]}(G_\Gamma)\]
Let $I$ be the ideal of $L_\Gamma$ generated by $w^{[2]}$. 

For each $u\in\euV(\Gamma)$ in the 2-restricted algebra $L_\Gamma$ one must have one of the following possibilities
\begin{equation}\label{eq:casesLie}\end{equation}
\begin{itemize} 
    \item either $\{u,w\}\in\euE$ so $[w,u]=0$ thus $[w^{[2]},u]=0$,
    \item or $(u,w)\in\euA$ so $[w,u]=u^{[2]}$ thus $[w^{[2]},u]=[w,w,u]=[w,u^{[2]}]=[u,u,w]=[u,u^{[2]}]=0$,
    \item or $(w,u)\in\euA$ so $[u,w]=w^{[2]}$ thus $[w^{[2]},u]=[w,w,u]=[w,w^{[2]}]=0$.
\end{itemize}
This means that $I$ is just the 2-restricted Lie algebra generated by $w^{[2]}$. 

Moreover, as $G_\Gamma/N$ is an almost 2-direct product, one may apply Lemma \ref{lem:FRp} to compute a presentation of
\[\gr_\bullet^{[2]}(G_\Gamma/N)=\gr_\bullet^{[2]}(T/N)_d\ltimes\gr_\bullet^{[2]}(HN/N)\]
where the derivation $d$ is given by $d(u)=0$ for $u\in\euV(\Gamma)$ such that either 
$\{u,w\}\in\euE$ or $(w,u)\in\euA$ and $d(u)=u^{[2]}$ if $(u,w)\in\euA$. Moreover, by induction one has 
\[\gr_\bullet^{[2]}(HN/N)=\gr_\bullet^{[2]}(H)=\gr_\bullet^{[2]}(G_\Delta)=L_\Delta\]
where $\Delta$ is the graph induced by $\euV(\Gamma)\setminus\{w\}$.
This means that 
\[\gr_\bullet^{[2]}(G_\Gamma/N)=L_\Gamma/I\]
and from this one deduces that there is a commutative diagram

$$\begin{tikzcd}
&I\arrow[r,hook]\arrow[d,"\tau_1"]&L_\Gamma\arrow[r,two heads]\arrow[d,"\tau"]&L_\Gamma/I\arrow[d,"\tau_2"]\\
&Z_\bullet\arrow[r,hook]&\gr_\bullet^{[2]}(G_\Gamma)\arrow[r,two heads]&\gr_\bullet^{[2]}(G_\Gamma/N)\\
\end{tikzcd}$$
where $\tau_2$ is an isomorphism and 
\[Z_\bullet=\oplus_{i\geq 1}\gamma_i^{[2]}(G_\Gamma)\cap N/\gamma_{i+1}^{[2]}(G_\Gamma)\cap N.\]
But as the subgroup $T=\langle w\rangle$ is strictly 2-embedded in $G_\Gamma$, we have that for any $i\geq 2$,
\[\gamma_i^{[2]}(G_\Gamma)\cap T=\gamma_i^{[2]}(T)\]
so
\[\gamma_i^{[2]}(G_\Gamma)\cap N=\gamma_i^{[2]}(G_\Gamma)\cap T\cap N=\gamma_i^{[2]}(T)\cap N\]
which means that $Z_\bullet$ can be computed explicitly. The same can be done with $I$ and one deduces that the map $\tau_1$ in the diagram is an isomorphism (for details, see the proof of Proposition 3.5 in \cite{ELMP} where a similar argument is used). Therefore, $\tau$ is also an isomorphism which yields the claim.
\end{proof}

\section{Twisted right-angled Artin pro-$p$ groups}
\label{s:genRAAG}
Let $\Z_p$ denote the \emph{ring of $p$-adic integers} and $v_p\colon\Z_p\to\Z\cup\{\infty\}$ its discrete valuation. Then $1+p\Z_p\subseteq\Z_p^\times$ coincides with the
 Sylow pro-$p$ subgroup of $\Z_p^\times$.

In \cite[(1.1)]{BQW} the authors defined for a prime number $p$, a continuous group homomorphism 
$\lambda\colon\Z_p\to\Z_p^\times$ and a finite mixed
graph  $\Gamma=(\euV,\euE,\euA)$ a finitely generated pro-$p$ group $G_{\Gamma,\lambda}$ and called it an \emph{oriented right-angled Artin pro-$p$ group}. In this section we generalise their definition slightly using a modification of \eqref{eq:defG} starting from an \emph{arrow $p$-labelled finite mixed graph}. In this context we call a finite mixed graph $\Gamma=
(\euV,\euE,\euA)$ together with a function
\begin{equation}
    \label{eq:arlabel}
    \ell\colon\euA\longrightarrow 1+p\Z_p
\end{equation}
to be an \emph{arrow $p$-labelled finite mixed graph}. For an arrow $p$-labelled finite mixed graph $(\Gamma,\ell)$ we define
\begin{equation}
\label{eq:orgenpropRAAG}   
G_{\Gamma,\ell}=\Bigg\langle\  \euV\  \Bigg\vert\ 
\begin{aligned}
[v,u]&=1\qquad\ \ \text{if 
$\{u,v\}\in\euE$,}\\
u^v&=u^{\ell(\eua)}\quad\text{if $\eua=(u,v)\in\euA$} \\
\end{aligned}
\Bigg\rangle_{\propp}
\end{equation}
and call $G_{\Gamma,\ell}$ the 
\emph{twisted right-angled Artin pro-$p$ group
associated to $(\Gamma,\ell)$}.
In case that 
$\Gamma$ has no arrows 
or if $\ell=\mathbf{1}$ is identically $1$, $G_{\Gamma,\ell}$ coincides
with the pro-$p$ completion of the right-angled Artin pro-$p$ group $G_{\ddGamma}$ based on the finite na\"ive graph
$\ddGamma=(\euV,\ddot{\euE})$.
The oriented right-angled Artin pro-$p$
group $G_{\Gamma,\lambda}$ coincides with the
twisted right-angled Artin pro-$p$ group
$G_{\Gamma,\ell}$, where the labelling function $\ell\colon\euA\to 1+p\Z_p$ is identically
$\lambda(1)$. Moreover, the pro-$2$ completion of the
twisted right-angled Artin group $G_{\Gamma}$ based
on the finite mixed graph $\Gamma$, coincides with
twisted right-angled Artin pro-$2$ group
$G_{\Gamma,\ell}$, where $\ell\colon\euA\to 1+2\Z_2$
is identically $-1$.

In \cite[Question~1.4]{BQW}) the authors formulated a question concerning the isomorphism type
of the restricted Zassenhaus $\F_p$-Lie algebra or equivalently its graded $\F_p$-algebra 
$\gr_\bullet\F_p\dbl G_{\Gamma,\lambda}\dbr$ associated to its completed $\F_p$-group algebra $\F_p\dbl G_{\Gamma,\lambda}\dbr$ in case that $\Gamma$ is a \emph{special} mixed graph and
$\image(\lambda)\subseteq 1+4\Z_2$ in case that $p=2$. 

Without going into detail it should be mentioned that
a special mixed graph $\Gamma=(\euV,\euE,\euA)$ does not contain oriented cycles, and therefore there exists a vertex $w\in\euV$ such that no arrow is starting in $w$, i.e., for all $u\in\euV$ one has $(w,u)\not\in\euA$. Moreover, defining $\Gamma_1,\Lambda\subseteq\Gamma$ and  as the finite mixed subgraphs induced by $\euV\setminus\{w\}$ and  $\{\,u\in\euV\mid (u,w)\in\euA\text{ or }\{u,w\}\in\euE\}$, respectively,
one has
\begin{equation}
\label{eq:deco3}
    G_{\Gamma,\ell}\cong G_{\Gamma_1,\ell}\ast^{(p)}_{H,\varphi,w}
\end{equation}
where $H=G_{\Lambda,\ell}$ and $\varphi\colon H\to H$
is the automorphism induced by $\varphi(v)=v^{\ell(v,w)}$. Here we put $\ell(u,w)=1$ if $\{u,w\}\in\euE$ is an edge. As for any prime $p$ one has $\ell(u,w)\in 1+p\cdot\Z_p$,  $\varphi$ is almost $p$-trivial.

\begin{theorem}
    \label{thm:propRAAG}
    Let $G_{\Gamma,\ell}$ be a twisted right-angled Artin pro-$p$ group based on the
    arrow $p$-labelled finite mixed graph $(\Gamma,\ell)=(\euV,\euE,\euA,\ell)$.  Assume further that $\Gamma$ has no oriented cycle and, in case that 
    $p=2$, that $\image(\ell)\subseteq 1+4\Z_2$. Then
    \begin{equation}
    \label{eq:propGR1}\gr_\bullet^{[p]}(G_{\Gamma,\ell})\cong(L_{\ddot{\Gamma}}\otimes_{\Z}\F_p)^{[p]}\end{equation}
    and $H^\bullet(G_{\Gamma,\ell},\F_p)\cong\Lambda^\bullet(\ddGamma^{\op})$.
    \end{theorem}
   

Here $L_{\ddGamma}$ denotes the right-angled Artin $\Z$-Lie algebra
with presentation
\begin{equation}
    \label{eq:RAAL}
    L_{\ddGamma}=\langle\, v\in\euV\mid [u,v] \text{ for } \{u,v\}\in\ddot{\euE}\,\rangle
\end{equation}
associated to the finite na\"ive graph 
$\ddGamma=(\euV,\ddot{\euE})$ (cf. \S\ref{s:TRAAG}),
and $(\argu)^{[p]}$ denotes the canonical restrictification functor
(cf \cite[Prop.1]{MRW}). By $\Lambda^\bullet(\Xi)$ we denote the Stanley-Reisner
algebra associated to a finite na\"ive graph $\Xi$.
Note that \eqref{eq:propGR1} implies that $\gr_\bullet(\F_p\dbl G_{\Gamma,\ell}\dbr)$
coincides with the Cartier-Foata algebra $A_{\ddGamma}\otimes_{\Z}\F_p$ based on the finite na\"ive graph $\ddGamma$ (cf. \cite{carf}). Thus, by a theorem of R.~Fr\"oberg, it is a Koszul $\F_p$-algebra - a property (cf. \cite{Prid}) which is expected to hold for all pro-$p$ groups occuring as maximal pro-$p$ Galois group $G_{\mathbb{K}}(p)$ of a field ${\mathbb{K}}$ (cf. \cite{W}). The second part of Theorem~\ref{thm:propRAAG}
then can be deduced from the first using the fact that if $\gr_\bullet\F_p\dbl G\dbr$ is Koszul, May's spectral sequence  collapses at the $E_1$-term.

\begin{proof}[Proof of Theorem~\ref{thm:propRAAG}]
It suffices to show \eqref{eq:propGR1}. As in the proof of Theorem~\ref{thm:Blum} we proceed by induction on $|\euV(\Gamma)|$. The induction hypothesis implies that for any proper induced subgraph $\Lambda$ of $\Gamma$, one has
$\langle\, x\in\euV(\Lambda)\,\rangle\cong
(L_{\ddot{\Lambda}}\otimes_\Z\F_p)^{[p]}$.
If $\euA=\emptyset$, then $G_{\Gamma,\ell}$ is the pro-$p$ completion of a right-angled Artin group, and there is nothing to prove 
(cf. \cite[Theorem~B and the examples thereafter]{MRW}). 
In the other case, taking into account that there is no oriented cycle in $\Gamma$, there is some $w\in\euV$ such that one has a decomposition
\eqref{eq:deco3}. By the induction hypothesis, \begin{align*}
    \gr_\bullet^{[p]}(G_{\Gamma_1,\ell})&\cong (L_{\ddot{\Gamma}_1}\otimes_{\Z}\F_p)^{[p]},\\
    \gr_\bullet^{[p]}(G_{\Lambda,\ell})&\cong(L_{\ddot{\Lambda}}\otimes_{\Z}\F_p)^{[p]}.
\end{align*}
Since $(L_{\ddot{\Lambda}}\otimes_{\Z}\F_p)^{[p]}$ is a retract of $(L_{\ddot{\Gamma_1}}\otimes_{\Z}\F_p)^{[p]}$, and
as there is canonical projection
$\tau\colon G_{\Gamma_1,\ell}\to G_{\Lambda,\ell}/cl(\gamma_2^{[p]}(G_{\Lambda,\ell}))$
satisfying the hypothesis (i) of Proposition~\ref{prop:crit}, $G_{\Lambda,\ell}$ is strictly $p$-embedded in $G_{\Gamma_1,\ell}$. Reasoning as in the proof of Theorem~\ref{thm:Blum} yields the relations
\begin{align}
[\mathfrak{v},\mathfrak{u}]&=0\qquad\ \ \text{if $\mathfrak{u}$ and $\mathfrak{v}$ are linked by an edge,}\notag\\
[\mathfrak{w},\mathfrak{u}]&=
\mathfrak{u}^{\ell(\eua)-1}\text{ if 
$\mathfrak{u}$ and $\mathfrak{w}$ are linked by an arrow $\eua\colon\mathfrak{u}\to \mathfrak{w}.$ }\notag
\end{align}
   For $p=2$, by hypothesis $u^{\ell(\eua)-1}\in G_{\Gamma,\ell}^4\subseteq\gamma_3^{[2]}(G_{\Gamma,\ell})$, while for $p$ odd,
   one has
   $u^{\ell(\eua)-1}\in 
   G_{\Gamma,\ell}^p\subseteq\gamma_3^{[p]}(G_{\Gamma,\ell})$ (cf. 
   \eqref{eq:Jen}. Thus by Corollary~\ref{cor:propHNN} one obtains
   $\gr_\bullet^{[p]}(G_{\Gamma,\ell})\simeq
   \gr_\bullet^{[p]}(G_{\Gamma_1,\ell})\ast_{\gr_\bullet^{[p]}(G_{\Lambda,\ell}),w,0}$ which yields the claim.
\end{proof}

\begin{remark}
    \label{rem:final}
    From Theorem~\ref{thm:propRAAG} one concludes that for a finite special graph $\Gamma$ and a continuous homomorphism $\lambda\colon\Z_p\to
    1+p.\Z_p$ (resp. $\lambda\colon \Z_2\to1+4\Z_2$
    if $p=2$), for the oriented right-angled Artin pro-$p$ group
    $G_{\Gamma,\lambda}$ one has
    $\gr_\bullet(G_{\Gamma,\lambda})\simeq
    (L_{\ddot{\Gamma}}\otimes_{\Z}\F_p)^{[p]}$
    answering \cite[Question~1.4]{BQW} affirmatively.
\end{remark}

The interested reader might wonder whether one may expect a result combining a large part of Theorem~\ref{thm:Blum} and Theorem~\ref{thm:propRAAG}. More precisely, one may ask whether the following question has an affirmative answer or not.

\begin{question}
\label{ques:final}
Let $(\Gamma,\ell)$ be a finite arrow labelled mixed graph. Assume further that $\Gamma$ has no oriented cycles. Is it true that
\begin{equation}
    \label{eq:final1}
    \gr_\bullet^{[p]}(G_{\Gamma,\ell})\simeq L_{\Gamma,\ell},
\end{equation}
where
\begin{equation}
\label{eq:defBlum2}
L_{\Gamma,\ell}=\Bigg\langle\  \euV\ \Bigg\vert\ 
\begin{aligned}
[v,u]&=0\qquad \text{if 
$\{u,v\}\in\euE$ or $p$ odd and $(u,v)\in\euA$,}\\
&\qquad\ \ \ \text{or $p=2$, $\eua=(u,v)\in\euA$ and $v_2(\ell(\eua)-1)\geq 2$,}\\
[v,u]&=u^{[2]}\quad\text{if $\eua=(u,v)\in\euA$, $p=2$,
$v_2(\ell(\eua)-1)=1$} \\
\end{aligned}\Bigg\rangle.
\end{equation}   
\end{question}

\end{document}